\documentclass[11pt,a4paper,english]{article}

\usepackage{amssymb,amscd, amsmath, babel, amstext, amsthm, latexsym}
\usepackage{color}

\newtheorem{remark}{Remark}

\numberwithin{equation}{section}
\numberwithin{defn}{section}
\numberwithin{thm}{section}
\numberwithin{lem}{section}
\numberwithin{remark}{section}

\newcommand{\mcal}{\mathcal}

\newcommand{\R}{\mathbb{R}}

\newcommand{\F}{\ensuremath{\mathcal{F}}}
\newcommand{\D}{\ensuremath{\mathcal{D}}}

\newcommand{\V}{\ensuremath{\mathcal{V}}}
\newcommand{\VV}{\ensuremath{\mathbb{V}}}
\newcommand{\W}{\ensuremath{\mathcal{W}}}

\newcommand{\A}{\ensuremath{\mathcal{A}}}
\newcommand{\B}{\ensuremath{\mathcal{B}}}
\newcommand{\T}{\ensuremath{\mathcal{T}}}
\newcommand{\DD}{\ensuremath{\mathbb{D}}}

\def \esssup{{\rm esssup}}

\def\INTER{\mathop{\rm {\cap}}\limits}

\def \BB{L_{\infty}({\cal B})}

\def\ben{\begin{enumerate}}
\def\een{\end{enumerate}}
\def\bit{\begin{itemize}}
\def\eit{\end{itemize}}
\newtheorem{definition}{Definition}
\numberwithin{definition}{section}
\newtheorem{theorem}{Theorem}
\newtheorem{lemma}[theorem]{Lemma}
\newtheorem{corollary}[theorem]{Corollary}
\newtheorem{proposition}[theorem]{Proposition}

\numberwithin{theorem}{section}

\title{Representation of convex operators and\\
their static and dynamic sandwich extensions}

\author{Jocelyne Bion-Nadal\thanks{UMR 7641 CNRS - Ecole Polytechnique. Ecole Polytechnique, 91128 Palaiseau Cedex, France. Email: jocelyne.bion-nadal@cmap.polytechnique.fr}\: and Giulia Di Nunno\thanks{Centre of Mathematics for Applications (CMA), Department of Mathematics,
University of Oslo, P.O. Box 1053 Blindern, N-0316 Oslo Norway.
Email: giulian@math.uio.no}
\thanks{Norwegian School of Economics and Business Administration (NHH), Helleveien 30, N-5045 Bergen, Norway.}}
\date{May 31, 2016}

\begin{document}
\maketitle

\begin{abstract}
Monotone convex operators and time-consistent systems of operators appear naturally in stochastic optimization and mathematical finance in the context of pricing and risk measurement. We study the dual representation of a monotone convex \emph{operator} when its domain is defined on a subspace of $L_p$, with $p\in [1,\infty]$, and we prove a sandwich preserving extension theorem. 
These results are then applied to study systems of such operators defined only on subspaces. 
We propose various dynamic sandwich preserving extension results depending on the nature of time: finite discrete, countable discrete, and continuous.
Of particular notice is the fact that the extensions obtained are time-consistent.
\end{abstract}

\section{Introduction}
The literature on extension theorems for functionals features some fundamental results. For all we just mention two: first is the Hahn-Banach theorem and its various versions, that provides e.g. a majorant preserving extension and then the K\"onig theorem that provides a sandwich preserving one (see e.g. \cite{FL}). Both cases give results for \emph{linear} functionals with values in $\mathbb{R}$.
This paper presents sandwich preserving extension theorems for \emph{convex} monotone operators defined in a subspace $L$ in $L_p(\B)$ with values in $L_p(\A)$ $(\A\subseteq \B)$, for $p\in [1,\infty]$.

Other results of this type are studied in the case of linear operators, see \cite{ADR} for $p\in[1,\infty)$ and \cite{BNDN} for $p=\infty$. Indeed the need for working in an operator setting taking values in $L_p(\A)$ is motivated by pricing and risk measurement in mathematical finance. To explain for any two fixed points in time, say $s \leq t$, a financial asset with payoff $X \in L_t \subseteq L_p(\F_t)$ has a price $x_{s,t}(X)$ evaluated at $s$. This value is $x_{s,t}(X) \in L_p(\F_s)$, where $\F_s \subseteq \F_t$.

These price operators are linear if the market model benefits of assumptions of smoothness, such as no transaction costs, no liquidity risk, perfect clearing of the market, no constrains in trading, etc. However, they are convex (when considered from the seller's perspective, the so-called ask-prices) when such market model assumptions are not fulfilled. 
Convex operators of such form also appear as value processes in the case of dynamic stochastic optimization and often this is in fact a way to obtain such price processes.

It is reasonable to have the domain of these operators defined on a subspace $L$ of the corresponding $L_p$ space. In fact, in general, not all positions are actually available for purchase in the market. It is only in the idealistic assumption of a complete market that we find that all positions are always feasible, i.e. the subspace is actually the whole $L_p$ space.
Strictly speaking, though less discussed in the literature, also the risk measurement is usually performed more reasonably on a subspace $L\subseteq L_p$. In this case $L$ represents those risks for which there is grounded measurement in terms of e.g. statistical knowledge, time series analysis, and general good information. For risks outside this set, one can always resolve with a conservative risk evaluation which corresponds to high (even too high to be competitive) hedging prices. The reasonable evaluation of risk and corresponding prices is relevant from an insurance perspective.

\bigskip
When dealing with a dynamic approach to pricing, we consider an information flow represented by a filtration $(\F_t)_{t \in [0,T]}$ $(T<\infty)$ and then a system of price operators is naturally appearing: $(x_{s,t})_{s,t\in\T}$, where $\T\subseteq [0,T]$. For the fixed times $s,t:\, s\leq t$ the price operator is $x_{s,t}: L_t \longrightarrow L_p(\F_s)$ where the domain is the subspace $L_t \subseteq L_p(\F_t)$.
An important necessary property of these system of operators is \emph{time-consistency}, which models the consistency of prices or measures of risk over time. 
Namely, for $s \leq t\leq u$ and $X \in L_u$, the evaluation $x_{s,u}(X)$ at time $s$ is required to coincide with the two steps evaluation  $x_{s,t}(x_{t,u}(X))$.

The question we address is how to extend the whole family of operators, so that the domains reach the whole $L_p(\F_t)$ in such a manner that time-consistency is preserved together with some sandwich property.
The \emph{sandwich property} itself is a control from above and below reasonably given on such operators, as it happens, in their own context, for the Hahn-Banach and K\"onig theorems. In applications this may assume various meanings. In \cite{AR}, \cite{ADR}, \cite{DE08} there are different studies on some aspects of the fundamental theorem of asset pricing with controls on tail events, first in a multiperiod market and then in a continuous time market.
In \cite{BNDN} the majorant and minorant operators are linked to no-good-deal dynamic bounds and the corresponding pricing measures. 
From the application perspective, the feasibility of these pricing rules is directly linked to the existence of the corresponding \emph{time-consistent sandwich preserving extension} of the system of price operators. 
So far this link has been explored only for linear pricing. The present paper provides fundamental results to address some questions related to convex pricing, e.g. no-good-deal bounds for convex pricing and its connections with indifferent pricing.

We stress that to obtain a time-consistent extension it is not enough to collect all the extensions of the single operators in one family. It is only via some careful procedure of extension that we can obtain such result.

\bigskip
Also it is important to mention that the \emph{representation result} we obtain for convex operators defined on a subspace $L$ of $L_p(\B)$ taking values in $L_p(\A)$ $(\A \subseteq \B)$ is crucial for the development of the extensions.
Representation theorems for convex operators have been studied in the context of risk-measures in the recent years. The first results were obtained for the static case, corresponding to operators with real values ($\A$ trivial). Here we have to mention \cite{FRG2002} (for $p \in [1,\infty]$) and \cite{FS04} (for $p=\infty$), where the domain of the operators is the whole $L_p(\B)$ space, and the paper \cite{BF2009}, where a very general framework is proposed, which also includes the case of a subspace $L$ of $L_p(\B)$ (for $p \in [1,\infty]$) with lattice property on $L$. In both cases the mappings take real values.
We also mention \cite{BN-preprint} and \cite{DS2005} for a representation in the case of operators defined on the whole $L_\infty(\B)$ with values in $L_\infty(\A)$ studied in the context of conditional risk measures.
Our contribution in this area provides a representation theorem for convex operators defined on $L \subseteq L_p(\B)$ (for $p \in [1,\infty]$) without requiring the lattice property of $L$.

\bigskip
In a summary our contribution provides various elements of novelty:
\begin{enumerate}
\item
A representation theorem for a convex operator defined on a subspace $L \subseteq L_p(\B)$ with values in $L_p(\A)$ with $\A \subseteq \B$ not trivial and without requiring the lattice property on $L$.
\item
A sandwich preserving extension theorem for such convex operators, where the sandwich bounds are  sublinear and superlinear operators defined on the whole $L_p(\B)$.
\item
Various time-consistent sandwich preserving extensions theorems dealing with a family of convex operators $x_{s,t}$ defined on $L_t \subseteq L(\F_t)$ with values in $L_p(\F_s)$, for $s \leq t$. We have to stress that
\begin{itemize}
\item
Time-consistency is a crucial property necessary in many applications such as in financial price evolution.
\item
Time-consistency is \emph{not} preserved by independent extensions of the single operators $x_{s,t}$.
\item
Time-consistent extensions require careful procedures in order to maintain the delicate relationships of the operators $x_{s,t}$ over time. 
\end{itemize}
\end{enumerate}

\bigskip
The paper is organized as follows. In Section 2 we give a precise presentation of the operators, the spaces, and the topology we  consider. Then the representation theorem is proved.
Section 3 is dedicated to the sandwich extension of such convex operators.
 The sequel of the paper deals with time-consistent systems of operators. In Section 4 the sandwich extension is studied in the case of discrete time. In Section 5 we reach out to obtain the sandwich extension for continuous time systems of operators.

\bigskip

\section{Convex operators in $L_p$ and their representation}

Let \((\Omega,\B,P)\) be a complete probability space.
Here we consider $\B$ to be the $P$-completed  $\sigma$-algebra generated by a countable family of sets in $\Omega$.
Also let $\A\subseteq \B$ be a $P$-augmented countably generated $\sigma$-algebra\footnote{This assumption will be implicitly used in the sandwich extension theorems. It is not necessary for the upcoming representation theorem.}.
For example, any Borel $\sigma$-algebra of a metrizable separable space completed by the $P$-null events satisfies this assumption.

For any $p \in [1,\infty]$ we consider the $L_p(\B):=L_p(\Omega,\B, P)$ of real valued random variables with the finite norms:
\[
\Vert X \Vert_p := 
\begin{cases}
(E[ \vert X\vert^p ] )^{1/p}, \qquad& p \in [1, \infty),\\
\esssup \vert X \vert, \qquad& p = \infty.
\end{cases}
\]
We equip these spaces with a topology.
In the cases $p \in [1, \infty)$, we consider the usual topology derived from the norm.
In the case $p=\infty$, we fix the weak* topology $\sigma(L_{\infty},L_1)$.
We will denote by the superscript ``+'', e.g. $L_p(\B)^+$, the cones of the non-negative random variables with the corresponding induced topology.

In the sequel we deal with a linear sub-space $L \subseteq L_p(\B)$.
We always consider $L$ equipped with the topology induced by the corresponding $L_p(\B)$ space.
Motivated by the applications we assume that:
\begin{enumerate}
\item[i)]
$1 \in L$,
\item[ii)]
for the $\sigma$-algebra $\A\subseteq \B$ we have the property that $1_A X \in L$ for every $A \in \A$ and every $X \in L$.
\item[iii)]
 For all $X \in L$ and every sequence  $X_n\in L$ such that $\sup_n||X_n||_p<\infty$  which converges $P$-a.s. to $X$, $X$ belongs to $L$.
\end{enumerate}

\bigskip
Fix $p\in [1,\infty]$ and the sub-space $L\subseteq L_p(\B)$ as above.
We consider an operator  
\begin{equation}
\label{operator}
x: L\longrightarrow L_p(\A)
\end{equation}
that is:
\begin{itemize}
\item
\emph{monotone}, i.e. for any $X',\,X''\,\in L$,
\begin{equation*}
  x(X')\geq   x(X''),\quad X'\geq X'',\label{mono}
\end{equation*}
\item 
\emph{convex}, i.e. for any $X',\,X''\,\in L$ and $\lambda\in [0,1]$, 
\begin{equation*}
  x \big(\lambda X'+ (1-\lambda) X''\big) \leq \lambda x(X')+ (1-\lambda) x(X'')
\label{addi}
\end{equation*}
\item 
\emph{Fatou property}, 
 for all sequence  $X_n$ in $L$ such that $\sup_n||X_n||_{L_p}<\infty$  which converges $P$.a.s. to $X$, 
\begin{equation}
\liminf_{n\to \infty} x(X_n) \geq x(X)
\label{lsc}
\end{equation}
\item 
\emph{ weak \(\A\)-homogeneous}, i.e. for all \(X\in L\)
\begin{equation*}
  x ( 1_A X)= 1_A   x(X),  \quad A \in \A,
\label{homo}
\end{equation*}
\item
\emph{projection property}
\begin{equation*}
\label{normal}
x (f)=f, \quad f\in L_p(\A)\cap L.
\end{equation*}
\end{itemize}
In particular we have $x(0)=0$ and $x(1)=1$.

\bigskip
Note that, if $p\in [1,\infty)$ and the operator $x$ is monotone and linear (as in \cite{ADR}) the assumption of weak $\A$-homogeneity is equivalent to $\A$-homogeneity, i.e. for all \(X\in L\), we have 
\begin{equation*}
  x (\xi X)= \xi   x(X)
\label{fullhomo}
\end{equation*}
for all  \(\xi \in L_p(\A)\) such that \(\xi X\in L\). 
If $p=\infty$ and the operator is linear and semi-continuous, then the same result holds (see \cite{BNDN}).
 

\bigskip
\bigskip
\subsection{Representation of a convex operator}
Our first result is a representation theorem for $L_p$-valued convex operators of the type above.
This can be regarded as a non-trivial extension of \cite[Theorem 5]{Rok}. The result by Rockafellar is written for functionals and can be retrieved setting $\A$ to be the trivial $\sigma$-algebra up to $P$-null events.

\begin{theorem}
\label{thm:rep-convex}
Let $x$ be an operator of the type \eqref{operator}.
Then the following representation holds: 
\begin{equation}
\label{eq:1}
x(X) = \esssup_{V \in \V} \Big\{ V(X) - x^*(V) \Big\}, \quad X \in L,
\end{equation}
where
\[
x^*(V) := \esssup _{X\in L} \Big\{ V(X) - x(X) \Big\}, \quad V\in \V,
\]
and $\V$ is the space of the linear, non-negative, continuous,  $\A$-homogeneous operators $V: L_p(\B) \longrightarrow L_p(\A)$ such that $E[ x^*(V)]<\infty$.

Moreover, the operator $x$ also admits representation in the form:
\begin{equation}
\label{eq:1-0}
x(X) = \esssup_{V \in \VV} \Big\{ V(X) - x^*(V) \Big\},  \quad X \in L,
\end{equation}
where
\begin{equation}
x^*(V) := \esssup _{X\in L} \Big\{ V(X) - x(X) \Big\}, \quad V\in \VV,
\label{eq:1-1}
\end{equation}
and $\VV$ is the space of the linear, non-negative, continuous,  $\A$-homogeneous operators $V: L_p(\B) \longrightarrow L_p(\A)$.
\end{theorem}

\bigskip
For future reference we borrow the terminology proper of the literature on risk measures and we call the operator $x^*$ \emph{minimal penalty}. 

\bigskip
Before the proof of the theorem we present a couple of technical lemmas.

\begin{lemma}
\label{lemma1}
If $V = 1_A V_1 + 1_{A^c} V_2$, for $V_1, V_2 \in \V$, $A \in \A$, and $A^c := \Omega \setminus A$, then 
\begin{equation}
\label{lattice}
x^*(V) = 1_A x^*(V_1) + 1_{A^c}x^*(V_2).
\end{equation}
Moreover the set $\big\{ V(X) - x^*(V): \; V\in \V \big\}$ is a lattice upward directed. 
\end{lemma}

\begin{proof}
For any $X,Y \in L$ we have 
\[
\begin{split}
1_A \big( V_1(X) - x(X) \big) + 1_{A^c} \big( V_2(Y) - x(Y) \big) 
&= V(1_A X +   1_{A^c} Y ) - x(1_A X +  1_{A^c} Y) \\
& \leq \esssup_{Z \in L} \big\{ V(Z) - x(Z) \big\}.
\end{split}
\]
Hence,
\[\begin{split}
1_A \esssup_{X \in L} \big\{ V_1(X) - x(X) \big\}+ 1_{A^c} &\esssup_{Y \in L} \big\{ V_2(Y) - x(Y) \big\}\\
&\leq \esssup_{Z \in L} \big\{ V(Z) - x(Z) \big\}.
\end{split}\] 
Namely, we have $1_A x^*(V_1) + 1_{A^c} x^*(V_2) \leq x^*(V)$.
On the other hand, for any $Z\in L$, we have
\[ \begin{split}
V(Z) -x(Z) &= 1_A \big( V_1(Z) -x(Z) \big) + 1_{A^c} \big( V_2(Z) -x(Z) \big) \\
& \leq 1_A \esssup_{Z\in L} \big\{ V_1(Z) - x(Z)  \big\} +  1_{A^c} \esssup_{Z\in L} \big\{ V_2(Z) - x(Z)  \big\}.
\end{split}\]
Therefore, $x^*(V) \leq  1_A x^*(V_1) + 1_{A^c} x^*(V_2)$. So \eqref{lattice} holds.

To prove the lattice property, it is enough to consider for any $V_1,V_2 \in \V$, the set $A:= \big\{ V_1(X) - x^*(V_1) \geq V_2(X) - x^*(V_2) \big\} \in \A$ and $V= 1_A V_1 + 1_{A^c}V_2$. From \eqref{lattice} we have that:
\[\begin{split}
V(X) - x^*(V) &= 1_A \big(  V_1(X) - x^*(V_1) \big) + 1_{A^c} \big(  V_2(X) - x^*(V_2) \big)\\
&= \sup \big\{  V_1(X) - x^*(V_1) , V_2(X) - x^*(V_2) \big\}.
\end{split}\]
By this we end the proof.
\end{proof}

\begin{lemma}
\label{lemma2}
For any $V \in \VV$, the set $\big\{ V(X) - x(X): \; X\in L \big\}$ is a lattice upward directed. 
\end{lemma}

\begin{proof}
We consider $X_1, X_2 \in L$ and we set $A:= \big\{ V(X_1) - x(X_1) \geq V(X_2) - x(X_2)\big\}\in \A$.
Consider $X= 1_A X_1 + 1_{A^c}X_2$. Then
\[\begin{split}
V(X) - x(X)  & = 1_A \big(  V(X_1) - x(X_1) \big) + 1_{A^c} \big(  V(X_2) - x(X_2) \big)\\
&= \sup \big\{  V(X_1) - x(X_1) , V(X_2) - x(X_2) \big\}.
\end{split}\]
By this we end the proof.
\end{proof}

\begin{lemma}
\label{lemma2b}
Let $x$ be an operator of type \eqref{operator} and consider
$h(X):=E\big[x(X)\big]$, $X\in L.$ 
Then $h$ is a convex form lower semi-continuous on $L$.
\end{lemma}

\begin{proof}
Let $c$ in $\R$, let  ${\cal C}=\{X \in L,\; h(X) \leq c\}$.\\
For $p < \infty$, we have to prove that ${\cal C}$ is closed in $L_p$. $L_p$ is  a Banach space, it is thus enough to consider a sequence $X_n$ in ${\cal C}$ with limit $X$ in $L_p$. We can assume that $\sup_n||X_n||_p < \infty$. There is a subsequence converging to $X$ $P$-a.s. It follows from the assumption iiii) on $L$ and the Fatou property that $X \in {\cal C}$.\\
In case $p = \infty$, we have to prove that ${\cal C}$ is closed for the weak* topology $\sigma(L_{\infty},L_1)$. The proof is inspired by the proof of Theorem 4.31 in \cite{FS04}
The set ${\cal C}$ being convex it is enough from the Krein Smulian Theorem (\cite{DSC} Theorem V 5 7) to prove that for all $r$ ${\cal C}_r={\cal C} \cap \{||X||_{\infty} \leq r\}$ is closed for the weak* topology. From Lemma A 64 in \cite{FS04} it is enough to prove that ${\cal C}_r$ is closed in $L_1$ for the norm topology. Let $X_n$ be a sequence in ${\cal C}_r$ with limit $X$ for the  $L_1$ norm. There is a subsequence of $X_n$ converging $P$ a.s. to $X$. Necessarily $||X||_{\infty}  \leq r$. It follows from  assumption iii) on $L$  that $X$ belongs to $L$.  From the Fatou property it follows that $X$ belongs to ${\cal C}_r$.
\end{proof}

\bigskip
We are now ready to prove Theorem \ref{thm:rep-convex}.
\begin{proof}
Set $$h(X) :=  E\big[ x(X)\big], \quad X \in L.$$
Clearly $h$ is a non-negative, monotone, convex, and proper (i.e., $h(X) > - \infty$ and finite for some $X$, see \cite[p. 1]{Rok}) functional. By Lemma \ref{lemma2b} it is also lower semi-continuous. Thus, by application of \cite[Theorem 4 and Theorem 5]{Rok}, we have the representation
\begin{equation}
\label{h}
h(X) = \sup_{v \in L^*} \Big\{ v(X) - h^*(v) \Big\}
\end{equation}
where
\[
h^*(v) :=  \sup_{X\in L} \Big\{ v(X) - h(X) \Big\}
\]
is the \emph{Fenchel transform of $h$} and $L^*$ is the set of continuous linear forms on $L$.
Note that we can restrict to $v\in L^*$ such that $h^*(v) < \infty$, and in this case $v$ is a non-negative linear form.
Recall that we always consider the usual $L_p$-norm topology in the case $p\in [1,\infty)$ and the weak* topology in the case $p= \infty$.
Then we distinguish the two cases.

If $p \in [1,\infty)$, the classical Hahn-Banach theorem guarantees that $v(X)$, $X\in L$, can be extended to a non negative continuous linear form $v(X)$, $X \in L_p(\B)$, and the extension admits the representation
\[
v(X) = E \Big[ fX \Big ], \quad X \in L_p(\B),
\]
for some $f \in L_q(\B)$ with $q= p(p-1)^{-1}$ and $f\geq 0$.

If $p = \infty$, then we refer to a version of the Hahn-Banach theorem for locally convex topological spaces as in \cite[Chapter II]{bourbaki2} and we proceed as follows.
Recall that the weak* topology on $L_\infty (\B)$, defined by the family of semi-norms
\[
p_g(\cdot) := E \Big[ g \cdot\Big], \quad  \textrm{ for every } g \in L_1(\B): \: g \geq 0,
\]
is locally convex.
For every non-negative linear form $v$ on $L$, continuous for the weak* topology, there is a semi-norm $p_g$ such that
\[
\vert v(X) \vert \leq p_g(X).
\]
Hence, applying the above mentioned corollary, we can extend $v$ to a non-negative weak* continuous linear form on $L_\infty(\B)$.
The extension admits the representation 
\[
v(X) = E \Big[ fX \Big ], \quad X \in L_\infty(\B),
\]
for some $f \in L_1(\B)$ such that $f\geq 0$.

Therefore for $p \in [1,\infty]$, the convex functional $h(X)$, $X\in L$, in \eqref{h} can be rewritten as:
\begin{equation}
\label{h2}
\begin{split}
h(X) &= \sup_{f\in L_q(\B): f \geq 0} \big\{ E[fX] - h^*(E[f\cdot] )\big\} \\
& = \sup_{f\in \W} \big\{ E[fX] - h^*(E[f\cdot] )\big\} 
\end{split}
\end{equation}
where 
\begin{equation}
\label{W}
\W :=\big \{ f\in L_q(\B): \; f\geq 0\;, h^*(E[f\cdot]) <\infty \big\}.
\end{equation}
Note that $\W \ne \emptyset$, because $h$ is real valued.

We remark immediately that $E[f\vert \A] = 1$, for every $f \in \W$. Indeed, consider any $A\in \A$ and $X= 1_A$. For any $\alpha\in \mathbb{R}$ we have
\[
E[ f \,\alpha 1_A] -h^*(E[f\cdot]) \leq h(\alpha 1_A) = E[x(\alpha 1_A)] = \alpha E[1_A].
\]
Hence,
\[
\alpha \big( E [f 1_A] - P(A) \big) \leq h^*(E[f \cdot]) < \infty.
\]
Thus taking $\alpha \to \pm \infty$, we see that $E[f1_A] = P(A)$, $A \in \A$. Namely $E[f\vert \A]=1$. 

For every $f\in \W$, denote 
\begin{equation}\label{V}
V(X) := E [fX \vert \A], \quad X \in L_p(\B).
\end{equation}
Hereafter we show that $V \in \V$.
First of all note that the operator $V$ is naturally non-negative, linear, and $\A$-homogeneous. It is also continuous. 
Indeed for the case $p\in [1, \infty)$ it is immediate from the conditional H\"older inequality.\\ For the case $p=\infty$, we consider a neighborhood of $E[fX\vert \A]$ for the weak* topology:
$$
\mcal{U}:= \big\{ Y \in L_\infty (\A) : \forall g_i \in L_1(\A),\, i= 1,...,h, \quad  \vert E\big[ g_i E[fX\vert \A \big] -E\big[ g_iY\big] \vert < \epsilon\big\}
$$
Since $f\geq0$ and $E[f\vert \A]=1$, then $g_i f\in L_1(\B)$, $i=1,...,h$, and the set
$$
\tilde{\mcal{U}}:= \big\{ Z \in L_\infty(\B) : \; 
\forall g_i \in L_1(\A),\, i= 1,...,h, \quad \vert E\big[ g_i f X \big] -E\big[ g_i f Z\big] \vert < \epsilon \big\}
$$
is a neighborhood of $X$ in $L_\infty(\B)$ in the weak* topology and for all $ Z \in \tilde{\mcal{U}}$,  $E[fZ \vert \A] \in \mcal{U}$.
This proves the continuity of $V$ for the weak* topology. Thus $V$ belongs to $\VV$.

Define  $x^*(V) := \esssup_{X \in L} \{ V(X) - x(X)\}$, for $V$ in \eqref{V}. We show that $E[x^*(V)]<\infty$. 
From the lattice property of Lemma \ref{lemma2}, from \cite[Prop VI 1.1.]{Neveu}, and the monotone convergence theorem, we have:
\[\begin{split}
E[x^*(V)] &= \sup_{X\in L} \big\{ E[V(X)] -E[x(X)] \big\}\\
&= \sup_{X\in L} \big\{ E[fX] -h(X) \big\}\\
&= \sup_{X\in L} \big\{ v(X)-h(X) \big\} = h^*(v)<\infty.
\end{split}\]
Then we can conclude that $V$ as in \eqref{V} is an element of $\V$.

We are now ready to prove the representation \eqref{eq:1}.
For every $V\in \mcal{V}$, define
\[
x^*(V) := \esssup_{X\in L} \Big\{ V(X) - x(X) \Big\}.
\]
Note that from $x(0)=0$, we have that $x^*(V) \geq 0$.
For every $V\in \mcal{V}$ and $X \in L$, we have
\[
x^*(V) \geq V(X) - x(X)
\]
or, equivalently,
\[
x(X) \geq V(X) - x^*(V).
\]
Thus
\begin{equation}
\label{eq:2}
x(X) \geq \esssup _{V\in \mcal{V}} \Big\{ V(X) - x^*(V) \Big\}.
\end{equation}
To conclude we need to show the reverse inequality.
To this aim it is enough to show that
\begin{equation}
\label{eq:3}
E\Big[x(X)\Big]\leq E\Big[ \esssup _{V\in \mcal{V}} \Big\{ V(X) - x^*(V) \Big\}\Big], \quad X \in L.
\end{equation}
Indeed we  have:
\[\begin{split}
E[x(X)] =h(X) &=  \sup_{f\in \W} \big\{ E[fX] - h^*(E[f \cdot]) \big\}\\
 &=  \sup_{f\in \W} \big\{E[ E[fX\vert \A]] - h^*(E[f \cdot]) \big\}\\
  & \leq  \sup_{V\in \V} \big\{ E[V(X)] - E[x^*(V)] \big\}\\
    &=  E[\esssup_{V\in \V} \big\{ V(X) - x^*(V) \big\} ]
\end{split}\]
where the last equality is due to the lattice property of Lemma \ref{lemma1} and \cite[Proposition VI.1.1]{Neveu}. We have then proved the representation \eqref{eq:1}.

To prove the representation \eqref{eq:1-0}, we note that $\V \subseteq \VV$. 
From equation \eqref{eq:1} we have
\begin{equation*}
\begin{split}
x(X) &= \esssup_{V \in \V} \Big\{ V(X) - x^*(V) \Big\} \\
 & \leq \esssup_{V \in \VV} \Big\{ V(X) - x^*(V) \Big\} .
  \end{split}
\end{equation*}
From the definition of $x^*$ in \eqref{eq:1-1} we have that
$
x^*(V) \geq V(X) - x(X)
$, that is $x(X) \geq V(X) - x^*(V)$, for every $V \in \VV$ and $X \in L$.
So, we conclude that
$$
x(X) \geq \esssup_{V \in \VV} \big\{ V(X) - x^*(V) \big\}.
$$
By this we end the proof.
\end{proof}

\bigskip
\begin{corollary}
\label{cor:1}
Let $x$ be of type \eqref{operator}. Then the following representation holds:
\begin{equation}
\label{eq:4}
x(X) = \esssup_{f \in \DD} \Big\{ E\big[ f X\vert \A \big ] - x^*(E\big[ f \cdot \vert \A \big ]  ) \Big\}, \quad X \in L,
\end{equation}
where
\begin{equation}
\label{D}
\DD :=\big \{ f\in L_q(\B): \; f\geq 0,\;  E[f \vert \A] =1 \big\}
\end{equation}
with $q=p(p-1)^{-1}$ 
and
\[
x^*(E\big[ f \cdot \vert \A \big ] ) := \esssup _{X\in L} \Big\{ E\big[ f X\vert \A \big ]  - x(X) \Big\}, \quad f \in \DD.
\]
\end{corollary}

\begin{proof}
From \cite[Theorem 1.1]{ADR}, for the case $p\in [1, \infty)$, and \cite[Theorem 3.3 and Proposition 3.14]{BNDN}, for $p = \infty$, we know that there is a one-to-one relationship between $\VV$ and $\DD$. Then the results follow directly from the representation \eqref{eq:1-0}.
\end{proof}

The representations of convex functionals were studied in the recent literature of risk measures in the case when the $\sigma$-algebra $\mathcal{A}$ is trivial.
In \cite{FRG2002} the representation is studied for convex risk measures (i.e. convex, monotone, lower semicontinuous, and translation invariant functionals) defined on the \emph{whole} $L_p(\B)$ with $p\in [1,\infty]$ with values in $(-\infty,\infty)$. In \cite{FS04} the study is carried on for $p=\infty$. In both cases it is crucial that the functionals are defined on the whole space.\\ 
In \cite{BF2009}, the representation is studied for convex, monotone, order continuous functionals defined on Fr\'echet lattices and taking values in $(-\infty,\infty]$. This allows for a general setup, however the vector space $L$ on which the map is defined needs to be a vector lattice.  This is not the case in the present paper. The lattice property in Lemma 2.1 and 2.2 is a lattice property in the arrival space $L_p({\cal A})$, not in $L$. Another important  difference is that in [4] the sigma-algebra ${\cal A}$ is trivial.  \\
If $\A$ is non-trivial, then we can refer to \cite{BN-preprint} and \cite{DS2005} for studies on the representation of convex operators in the context of conditional risk measures (i.e. convex, monotone, lower semicontinuous, translation invariant operators) defined on the \emph{whole} $L_\infty(\B)$.\\
In conclusion, our contribution to this area provides a representation of convex operators defined on a subspace $L \subseteq L_p(\B)$ with values in $L_p(\A)$, $p\in [1,\infty]$, with very mild hypothesis on $L$.

\bigskip
\bigskip
\section{Sandwich extension of a convex operator}
In the sequel we consider a criterion for the existence of an extension $\bar x$ of the convex operator $x$ to the whole $L_p(\B)$. The given $x$ lies within two operators $m$ and $M$. This extension is sandwich preserving. There is no uniqueness of such sandwich preserving extension, but our approach allows for an explicit representation of at least one of them, denoted $\hat x$, which turns out to be the maximal.

First of all we introduce the \emph{minorant} as a superlinear operator:
$
m: L_p(\B)^+ \longrightarrow L_p(\A)^+
$, i.e.,
\begin{equation*}
\label{superlinear}
m(X+Y) \geq m(X) + m(Y) , \qquad  X,Y \in L_p(\B)^+, 
\end{equation*}
$$
m(\lambda X) = \lambda m(X), \qquad X \in L_p(\B)^+, \, \lambda \geq 0,
$$
and the \emph{majorant} as a sublinear operator:
$
M: L_p(\B)^+ \longrightarrow L_p(\A)^+
$, i.e., 
\begin{equation*}
\label{sublinear}
M(X+Y) \leq M(X)+M(Y), \qquad  X,Y \in L_p(\B)^+, 
\end{equation*}
$$M(\lambda X)=\lambda M(X), \qquad X \in L_p(\B)^+,\, \lambda \geq 0 .$$
We remark that sublinearity implies $M(0)=0$.

Moreover, in the case $p=\infty$, we say that the map $M:L_{\infty}(\B)^+ \longrightarrow  L_{\infty}(\A)^+$ is \emph{regular} if for every decreasing sequence $(X_n)_n$ in $\BB$ with $X_n \downarrow 0$, $n\to \infty$ $P$-a.s, we have
\begin{equation}
\label{regular}
M(X_n) \rightarrow 0, \quad n \to \infty\quad P a.s. 
\end{equation}

\begin{theorem}
Fix $p \in [1, \infty]$. Let $x:L \longrightarrow L_p({\cal A})$ be of type \eqref{operator}. Consider 
the weak $\A$-homogeneous operators $m,M:  L_p({\cal B})^+ \rightarrow  L_p({\cal A})^+$ such that $m$ is  superlinear and $M$ is sublinear and, if $p=\infty$, $M$ is also regular.
Assume the sandwich condition:
\begin{equation}
m(Z)+x(X) \leq M(Y)
\label{eqS1}
\end{equation}
$$\forall X \in L\quad \forall \;Y,Z \in L_p({\cal B})^+ :\;\; Z+X \leq Y.$$ 
Then $x$ admits an extension $\hat x$ (to the whole  $L_p({\cal B})$), which is convex, monotone, lower-semicontinuous, weak ${\cal A}$-homogeneous and satisfying the projection property such that \eqref{eqS1} is preserved, i.e. 
\begin{equation}
 m(Z)+\hat x(X) \leq M(Y) 
\label{eqS2}
\end{equation}
$$\forall X \in L_p({\cal B})\quad \forall \;Y,Z \in L_p({\cal B})^+ :\;\; Z+X \leq Y.$$
In particular the operator
\begin{equation}
\hat x (X):= \esssup_{V\in  \V^S} \big\{ V(X) - x^*(V) \big\}, \quad X \in L_p(\B),
\label{ext-S1}
\end{equation}
with
\begin{equation}
 x^* (V):= \esssup_{X\in L} \big\{ V(X) - x(X) \big\}, \quad V \in \V^S,
\label{pen-S1}
\end{equation}
is a sandwich preserving extension of $x$.
Here above ${\cal \V^S}$ is the set of linear, continuous, non-negative, ${\cal A}$-homogeneous operators on $L_p({\cal B})$ such that $E[x^*(V)]<\infty$, and satisfying the sandwich condition: 
$$
m(X) \leq V(X) \leq M(X), \qquad  X \in L_p({\cal B})^+.
$$
Moreover, for any other such extension ${\bar x}$, we have that
$$
\hat x(X) \geq {\bar x}(X), \qquad X\in L_p(\B).
$$
We call $\hat x$ the maximal extension.
\label{thmsand}
\end{theorem}

\begin{proof}
From Theorem \ref{thm:rep-convex}, for all $X \in L$,
$$
x(X)=\esssup_{V\in {\cal V}} \big\{ V(X) - x^*(V) \big\}.
$$
Thus $\forall V \in {\cal V}$, the restriction of $V$ to $L$ satisfies:
$$
m(Z)+V(X)-x^*(V) \leq M(Y)\;\; \forall X \in  L \; \forall Y,Z \in L_p({\cal B})^+: \;\; Z+X \leq Y.
$$ 
Then, for all $\alpha > 0$, 
$$\alpha m(Z)+ \alpha V(X)-x^*(V) \leq \alpha M(Y)$$
Let $A=\{m(Z)+V(X)-M(Y) \geq 0\}$, 
$$0 \leq E(1_A (m(Z)+V(X)-M(Y)) \leq \frac{1}{\alpha} E(1_A x^*(V))<\infty.$$
Let $\alpha \rightarrow \infty$, it follows that $1_A (m(Z)+V(X)-M(Y))=0\; P\;a.s.$. 
Thus 
\begin{equation}
\label{sand-V}
m(Z)+V(X) \leq M(Y)
\end{equation} 
for all $V \in {\cal V}$ and $\forall X \in  L,  \forall Y,Z \in L_p({\cal B})^+:$ $ Z+X \leq Y$.
From the sandwich extension theorem for linear operators, \cite[Proposition 3.11]{BNDN} in case of $L_p$ spaces $1 \leq p <\infty$, and \cite[Theorem 3.9]{BNDN} in case of $L_{\infty}$ spaces (see also \cite[Theorem 5.1]{ADR}), every $V \in {\cal V}$ restricted to $L$, admits a sandwich preserving linear extension to the whole $L_p({\cal B})$ denoted $V^S$ which is monotone, lower semi continuous, weak ${\cal A}$-homogeneous, and satisfying  the sandwich condition: 
$$
m(Z)+V^S(X) \leq M(Y) \;\;\forall X \in L_p(\B)\; \forall  Y,Z \in L_p({\cal B})^+:\; \; Z+X \leq Y.
$$
Define 
$$
\hat x(X) :=\esssup_{V\in \V^S} \big\{ V(X) - x^*(V) \big\}, \qquad X\in L_p(\B),
$$
where $\V^S$ is the set described in the statement of the theorem, and $x^*(V)=\esssup_{Y \in L} (V(Y)-x(Y))$. 
It follows that $\hat x$ extends $x$ and it is lower semi continuous, convex, monotone, weak ${\cal A}$-homogeneous and it also satisfies the projection property.
It remains to verify that $\hat x$ satisfies the sandwich condition.
Let $Y,Z \in L_p({\cal B})^+\; \forall X \in L_p(\B):\; Z+X \leq Y$,
\begin{eqnarray}
m(Z)+ \hat x(X)& = &\esssup_{V\in \V^S} \big\{ m(Z)+V(X) - x^*(V) \big\}\nonumber \\& \leq  & M(Y)+\esssup_{V\in \V^S}(-x^*(V))\nonumber \\& =  &  M(Y)+
\esssup_{V\in {\cal V}}(-x^*(V))=M(Y).
\end{eqnarray}

Now consider any other convex, monotone, lower-semicontinuous, weak $\A$-homogeneous extension ${\bar x}$ satisfying the sandwich condition. 
From Theorem \ref{thm:rep-convex} we have that 
$$
{\bar x}(X) = \esssup_{V \in \V_{{\bar x}}}  \big\{ V(X) -  {\bar x}^*(V)\big\}, \quad X \in L_p(\B),
$$
with
$$
{\bar x}^*(V) = \esssup_{X \in L_p(\B)}  \big\{ V(X) - {\bar x} (X) \big\}
$$
where $\V_{{\bar x}}$ is given by the mentioned theorem with reference to the operator ${\bar x}$.
Moreover, since ${\bar x}$ satisfies the sandwich condition we can see that \eqref{sand-V} holds for $\V_{{\bar x}}$ and that 
$$
{\bar x}(X) = \esssup_{V \in \V_{{\bar x}}^S} \big\{ V(X) - {\bar x}^*(V)\big\}, \quad X \in L_p(\B).
$$
From the definition of ${\bar x}^*$ and of $x^*$ with ${\bar x}(X) = x(X)$, $X\in L$, we can see that ${\bar x}^*(V) \geq x^*(V)$ is valid for all $V \in \VV$.  Hence $E[\bar x^*(V)] \geq E[x^*(V)]$, $V \in \VV$, and in particular  $\V_{{\bar x}}^S \subseteq \V^S$. 
Then
${\bar x}(X) \leq \esssup_{V \in \V^S} \big\{ V(X) - {\bar x}^*(V)\big\}$.
On the other hand for every $ {V \in \V^S}$ and $X \in L_p(\B)$ we have 
$V(X) - {\bar x}(X) \leq {\bar x}^*(V)$, hence $V(X) - {\bar x}^*(X) \leq {\bar x}(X)$.
Thus we conclude $\esssup_{V \in \V^S} \big\{ V(X) - {\bar x}^*(V)\big\} \leq {\bar 
x}(X)$ and we have proved that:
$$
{\bar x}(X) = \esssup_{V \in \V^S} \big\{ V(X) - {\bar x}^*(V)\big\}.
$$
Since 
${\bar x}^*(V) \geq x^*(V)$ for all $V \in \V_{{\bar x}}^S$, then 
${\bar x}(X) \leq \hat x(X)$ for all $X \in L_p(\B)$.
\end{proof}

\begin{remark}
The above extension $\hat x$ \eqref{ext-S1} satisfies the following inequality:
$$ 
\forall X \in L_p({\cal B})^+ \quad m(X) \leq -\hat x(-X) \leq \hat x(X) \leq M(X).
$$
This inequality is in fact equivalent to \eqref{eqS2} for every convex, monotone, lower semi continuous, weak ${\cal A}$-homogeneous operator defined on the whole $L_p({\cal B})$. 
\rm{The first assertion   follows from equation \eqref{eqS2} applied one time with $(Z,X,Y)=(X,-X,0)$ and the other time with $(Z,X,Y)=(0,X,X)$. The second assertion follows from the convexity of $\hat x$.
}
\label{rem0}
\end{remark}

\begin{corollary}
\label{cor:S1}
For every $V \in \V^S$, the penalty \eqref{pen-S1} in the representation \eqref{ext-S1} of the extension $\hat x$ of the operator $x$, satisfies $x^*(V) = \tilde x(V)$, where
\begin{equation}
\label{pen-S2}
\tilde x (V):= \esssup_{X\in  L_p(\B)} \big\{ V(X) - \hat x(X) \big\}.
\end{equation} 
Moreover, define $\VV^S$ as the set of elements in $\VV$ satisfying the sandwich condition \eqref{sand-V}. Then the extension \eqref{ext-S1} can be rewritten as:
\begin{equation}
\hat x (X) = \esssup_{V\in  \VV^S} \big\{ V(X) - \tilde x(V) \big\}, \qquad X\in L_p(\B).
\label{ext-S2}
\end{equation}
Furthermore, we can also give the representation:
\begin{equation}
\hat x (X) = \esssup_{f\in  \DD^S} \big\{ E[fX\vert \A ] - \tilde x(E[f\cdot \vert \A]) \big\}, \qquad X\in L_p(\B),
\label{ext-S3}
\end{equation}
with 
\begin{equation}
\label{DS}
\DD^S:= \Big\{ f  \in \DD: \, m(X) \leq E[fX\vert \A ] \leq M(X), \: \forall X \in L_p(B) \Big\}.
\end{equation}
The penalty $\tilde x$ is called {\it minimal penalty} following the terminology of risk measures.
\end{corollary}
\begin{proof}
Fix $V \in \V^S$. 
From the definition we have $x^*(V) \leq \tilde x(V)$. On the other hand, from \eqref{ext-S1}, we have $\hat x(X) \geq V(X) - x^*(V)$, for all $X \in L_p(\B)$.
Hence, $ x^*(V) \geq V(X) - \hat x(X)$, for all $X \in L_p(\B)$. 
Thus  $x^*(V) \geq \tilde x(V)$.
This proves \eqref{pen-S2}.
For the proof of \eqref{ext-S2} we apply the same arguments as for the proof of \eqref{eq:1-0} in Theorem \ref{thm:rep-convex}.
And for the proof of \eqref{ext-S3} we apply the same arguments as for the proof of \eqref{eq:4} in Corollary \ref{cor:1}.
\end{proof}

\begin{definition}
The operator $m$ is \emph{non degenerate} if it satisfies $E(m(1_B))>0$ for all $B \in {\cal B}$  such that $P(B)>0$.
\label{defnondeg}
\end{definition}

\begin{lemma}
\label{lemmaequiv0}
Assume that $m$ is non degenerate.  Every $f \in  \DD^S$ such that $E(\tilde x(E(f \cdot |{\cal A}))<\infty$ belongs to 
$$ 
\DD^e :=\{f \in  \DD\;|f>0 \; P a.s.\}.
$$
\end{lemma}
\begin{proof}
 Let $B \in {\cal B}$ such that $P(B)>0$. It follows from the Remark \ref{rem0}  that  for all real $\lambda>0$, $\hat x (- \lambda 1_B) \leq -m(\lambda 1_B)$. From the representation (\ref{ext-S3}) of $\hat x (-\lambda 1_B)$, we get $\hat x(- \lambda 1_B)  \geq E(- \lambda 1_Bf)-\tilde x(E(f \cdot |{\cal A}))$. It follows that for all $\lambda>0$, $E(1_Bf) \geq E(m(1_B))- \frac{E(\tilde x(E(f.|{\cal A}))}{\lambda}$. Letting $\lambda \rightarrow \infty$, the result follows from $E(m(1_B))>0 $, being $m$ non degenerate.
\end{proof}

We deduce the following result from Corollary \ref{cor:S1} and Lemma \ref{lemmaequiv0}.
\begin{corollary}
\label{cor:S1b}
Assume that $m$ is non degenerate, then $\hat x$ admits the following representation
\begin{equation}
\hat x (X) = \esssup_{f\in  \DD^{S,e}} \big\{ E[fX\vert \A ] - \tilde x(E[f\cdot \vert \A]) \big\}, \qquad X\in L_p(\B),
\label{ext-S3b}
\end{equation}
with 
\begin{equation}
\label{DSe}
\DD^{S,e}:= \DD^{S} \cap \DD^{e}.
\end{equation}
\end{corollary}

The following result can be regarded as an extension of \cite[Theorem 5.2]{ADR} to the case of convex operators.

\begin{corollary}
\label{cor:S2}
If the minorant $m$ and the majorant $M$ in Theorem \ref{thmsand} are linear operators:
\[\begin{split}
m(X) &= E\big[ m_0X \vert \A\big], \quad X\in L_p(\B)^+,\\
M(X) &= E\big[ M_0X \vert \A\big], \quad X\in L_p(\B)^+
\end{split}\]
for some random variables $m_0, M_0 \in L_q(\B)^+$: $q= p(p-1)^{-1}$ such that $0\leq m_0 \leq M_0$.
The extension \eqref{ext-S1} $\hat x$ can be written as:
\begin{equation}
\hat x (X) = \esssup_{f\in  \D} \big\{ E[fX\vert \A ] -  x^*(E[f\cdot \vert \A]) \big\}, \qquad X\in L_p(\B),
\label{ext-S4}
\end{equation}
where
$$
\D:= \Big\{ f  \in L_q(\B): \, 0 \leq m_0 \leq f \leq M_0, E[f\vert \A]=1 \Big\}.
$$
\end{corollary}
\begin{proof}
This is a direct application of Corollary \ref{cor:S1}.
\end{proof}

\bigskip
\bigskip

We now prove that under the sandwich condition the $\esssup$ in \eqref{ext-S1} is attained. This will be a consequence of the following compactness result.

\begin{lemma}
Let $M$ be  sublinear, monotone, weak ${\cal A}$-homogeneous,  and, if $p=\infty$, regular.  Let  ${\cal K}=\{f \in  \DD: \;\;0 \leq E(f \cdot |{\cal A}) \leq M \}$. Identifying $ f \in {\cal K}$ with the linear form $E(f \cdot)$ on $L_p(\mathcal{B})$, ${\cal K}$  is a compact subset of the  ball of radius $E(M(1)^q)^{\frac{1}{q}}$ of  $L_p'(\B)$, $1 \leq p  \leq  \infty$ equiped with the weak* topology  $\sigma (L_p',L_p)$. In case $p=\infty$, ${\cal K}$ is furthermore contained in $L_1(\mathcal{B})$.(Notice that if $p<\infty$, $L_p'=L_q$ with $q=p(p-1)^{-1}$.) \\
Moreover, with the notations of  Theorem \ref{thmsand}, the set $\DD^S$ is compact for the weak* topology .
\label{lem.tech}
\end{lemma}
\begin{proof}
${\cal K}$ is a subset of the ball of radius $E(M(1)^q)^{\frac{1}{q}}$ of $L_{p}'(\B)$. As this bounded  ball is compact for the weak* topology (Banach Alaoglu theorem), it is enough to prove that ${\cal K}$ is closed for the weak* topology. Denote $\overline {\cal K}$  the weak* closure of ${\cal K}$. Let $\Psi \in \overline {\cal K}$. $\Psi$ is a positive continuous linear form on $L_p(\B)$. 

In case $p\in [1,\infty)$, $\Psi$ is represented by an element of $L_q(\B)$ for $q=p(p-1)^{-1}$ (Riesz representation theorem). \\
We detail the case $p=\infty$. 
We first prove that   $\Psi$ defines a  measure on $(\Omega,{\cal B})$. Let $(X_n)_n$ be any  sequence  of elements of $L_{\infty}(\B)$ decreasing to $0$ $P\; a.s$. From the regularity of $M$, $\forall \epsilon >0$, there is $n_0$ such that $\forall n \geq n_0$, $E(M(X_n)) \leq \epsilon$.   Denote $U$ the neighborhood of $\Psi$ defined as $U =\{\phi \in L_{\infty}'(\B), \;\;|\Psi(X_{n_0})-\phi(X_{n_0})| \leq \epsilon\}$. 
Since $\Psi \in \overline {\cal K}$, there is $\phi \in U \INTER {\cal K}$. For such $\phi$, $0 \leq \phi(X_n) \leq E(M(X_n))\leq \epsilon$. It follows that $|\Psi(X_{n_0})| \leq 2 \epsilon$. As $\Psi$ is a non negative linear functional and the sequence $(X_n)_n$ is decreasing to $0$, it follows that $0 \leq \Psi(X_{n}) \leq 2 \epsilon$ for every $n \geq n_0$. From Daniell Stone Theorem, see e.g. \cite[Theorem A48]{FS04}, it follows that $\Psi$ defines a probability measure on $(\Omega,{\cal B})$. This probability measure is absolutely continuous with respect to $P$ and this gives the existence of some $g \in L_1 (\mathcal{B})$ such that $\Psi=E(g \cdot)$ (Radon Nikodym Theorem). For all $X \in L_{\infty}({\cal A})$, the equality $\Psi(X)=E(X)$ is obtained similarly making use of the neighborhood of $\Psi$: $ U_X=\{\phi \in L_{\infty}'(\B), \;\;|\Psi(X)-\phi(X)| \leq \epsilon\}$. It follows that $E(g|{\cal A})=1$.
The inequality $E(fX1_A)\leq E(M(X)1_A)$ for $X \in L_p({\cal B})$ and $A \in {\cal A}$ is obtained similarly and hence $\Psi=E(g.)$ where $g$ belongs to ${\cal K}$. This proves the compactness of ${\cal K}$ for the weak* topology. \\
$\DD^S$ is equal to $\{f \in {\cal K}: E(m(X)1_A) \leq E(fX1_A), \forall X \in L_p({\cal B)},\forall A \in {\cal A}\}$. Thus $\DD^S$  is a  closed subset of  ${\cal K}$ for the weak* topology.
\end{proof}

\begin{proposition}\label{propmax}
Assume the hypothesis of Theorem \ref{thmsand}.
For every $X \in L_p({\cal B})$, there is some $f_X$ in $\DD^S$ (depending on $X$) such that 
\begin{equation}
\hat x(X)=E(f_X X |{\cal A})-\tilde x (E(f_X \cdot|{\cal A})).
\label{eqmax}
\end{equation}
\end{proposition}
\begin{proof}
We start from the representation (\ref{ext-S3}) given in Corollary \ref{cor:S1}:
\begin{equation}
\hat x (X) = \esssup_{f\in  \DD^S} \big\{ E[fX\vert \A ] -  \tilde x (E[f \cdot \vert \A]) \big\}, \qquad X\in L_p(\B).
\label{ext-S01}
\end{equation}
From the lattice property proved in Lemma \ref{lemma1}, it follows that 
$E(\hat x (X))= \sup_{f\in  \DD^S}[ E(fX) -  E(\tilde x (E[f\cdot \vert \A])])$.
From the definition of $\tilde x$ and the lattice property proved in Lemma \ref{lemma2}, it follows that $ E(\tilde x (E[f\cdot \vert \A]))$ is a lower semi continuous function of $f\in \DD^S$ for the weak* topology and thus we deduce from the compactness of $\DD^S$ (see Lemma \ref{lem.tech}) that the upper semi continuous function  $E(fX) -  E(\tilde x(E[f\cdot \vert \A]))$, $f\in \DD^S$, has a maximum attained for some $f_X$ (which may not be unique). From equation  (\ref{ext-S01}) it then follows that $f_X$ satifies (\ref{eqmax}).
\end{proof}

\section{Sandwich extensions of discrete time systems }

We equip the probability space \((\Omega,\B,P)\) with the right-continuous $P$- augmented filtration $(\F_t)_{t\in [0, T]}$.
We assume that, for all $t$, $\F_t$ is generated by a countable family of events, by which we mean that $\F_t$ is the smallest $\sigma$-algebra containing the countable family and all $P$-null events.

\vspace{1mm}
Let $p\in [1,\infty]$. For any time \(t\in[0,T]\) $(T>0)$, consider the linear sub-space:
\begin{equation}
L_t\subseteq L_p(\F_t) , \quad L_t\subseteq L_T.
\end{equation}
Let $\T \subseteq [0,T]$ such that $0,T\in \T$.
In the sequel we denote $(x_{s,t})_{s,t\in \T}$ on $(L_t)_{t\in \T}$ the system of operators $x_{s,t}:L_t \longrightarrow L_s$ of the type \eqref{operator}, for $s,t\in \T$: $s\leq t$.

\vspace{1mm}
In financial applications these operators represent a time-consistent system for ask prices in a market with friction. 
The time $s$ is the price evaluation time of an asset which has payoff at $t$ and the prices are defined on the domain $L_t$ of purchasable assets. 
Note that, in general, \(L_t \subset L_p(\F_t)\) for some \(t\in [0,T]\) and $L_t=L_p(\F_t)$ for all $t\in[0,T]$ in a complete market.

\begin{definition} \label{pcons}
The system $(x_{s,t})_{s,t \in \T}$, is \emph{time-consistent} (or \emph{$\T$ time-consistent}) if for all $s, t, u\in \T$: $s \leq t \leq u$
\begin{equation}\label{consist}
x_{s,u}(X)=x_{s,t}\big(x_{t,u}(X)\big),
\end{equation}
for all $X\in L_u$.
\end{definition}

Time-consistency is a natural assumption for such system of operators representing, e.g., price processes.
This concept models the reasonable equivalence of the price evaluation for an asset with payoff at time $u$, say, when the evaluation is performed either in one step, i.e. the straight evaluation of the asset at time $s$, or in two steps, i.e. first an evaluation at time $t: \,t\leq u$ and then at $s:\,s \leq t \leq u$.
This concept is also proper of a consistent risk measurements and it is studied for dynamic risk measures (where it is called {\it strong} time-consistency in \cite{AP}), see e.g. \cite{D}, \cite{BN01}.

\begin{remark}
\label{4.1}
For any $s\leq t \leq T$, $x_{st}$ is the restriction to $L_t$ of $x_{sT}$. 
\end{remark}
Indeed let $X \in L_t$, then $x_{tT}(X)=X$, by the projection property. Thus  by time-consistency we have 
$x_{sT}(X)=x_{st}(x_{tT}(X))=x_{st}(X)$, for all $X\in L_t$.

\vspace{3mm}
In the sequel we discuss extension of dynamic systems of operators which will be sandwich preserving. We deal with systems of superlinear and sublinear operators: each one representing the minorant and majorant of one of the  operators to be extended. 
Motivated by applications, a modification of the concept of time-consistency is also necessary.
Examples of studies of such minorants and majorants are found in \cite{ADR}, \cite{DE08}, and \cite{BNDN}. It is in this last paper that  the general concept of \emph{weak time-consistency} is introduced for the first time in connection with no-good deal bounds. We are now considering again this general definition in this context of convex operators also in view of upcoming applications to the study of ask prices in the context of risk-indifference pricing. 

\begin{definition}
\begin{itemize}
\item 
The family $(m_{s,t})_{s,t\in\T}$ of weak $\mcal{F}_s$-homogeneous, superlinear operators $m_{s,t}: L_{p}({\cal F}_t)^+ \rightarrow L_{p}({\cal F}_s)^+$ is \emph{weak time-consistent} if, for every $X \in L_{p}({\cal F}_t)^+$, 
\begin{equation}
m_{r,s}(m_{s,t}(X)) \geq m_{r,t}(X),\quad\forall r \leq s \leq t.
\label{wtcm}
\end{equation}
\item 
The family $(M_{s,t})_{s,t\in\T}$ of weak $\mcal{F}_s$-homogeneous, sublinear operators $M_{s,t}: L_{p}({\cal F}_t)^+ \rightarrow L_{p}({\cal F}_s)^+$ is \emph{weak time-consistent} if, for every $X \in L_{p}({\cal F}_t)^+$, 
\begin{equation}
M_{r,s}(M_{s,t}(X)) \leq M_{r,t}(X), \quad \forall r \leq s \leq t.
\label{wtcM}
\end{equation}
\label{defwtc}
\end{itemize}
\end{definition}
Note that the operators $m_{s,t}$, $M_{s,t}$ are not required to satisfy the projection property. 

\begin{definition}
\label{mM1}
We say that the family $(m_{s,t}, M_{s,t})_{s,t\in \T}$ satisfies the \emph{mM1-condition} if they are weak time-consistent families of superlinear, respectively sublinear, weak $\F_s$-homogeneous operators such that  \,$m_{s,t}, \, M_{s,t}: L_p(\F_t)^+ \longrightarrow L_p(\F_s)^+$, $m_{0,T}$ is non degenerate, and $M_{s,t}$ is also regular if $p=\infty$.
\end{definition}

\begin{definition}
We say that the system of operators $(x_{s,t})_{s,t\in \T}$ satisfies the \emph{sandwich condition} when
\begin{equation}
m_{s,t}(Z)+x_{s,t}(X) \leq M_{s,t}(Y)
\label{dynamic sand}
\end{equation}
$$\forall X \in L_t\quad \forall \;Y,Z \in L_p({\cal F}_t)^+ :\;\; Z+X \leq Y,$$
for some families of operators $(m_{s,t})_{s,t\in \T}$ and $(M_{s,t})_{s,t\in \T}$ with $m_{s,t}, \, M_{s,t}: L_p(\F_t)^+ \longrightarrow L_p(\F_s)^+$.
\end{definition}

\subsection{Finite discrete time systems}

First of all we consider a finite set $\T: = \{s_1,...,s_K:\: 0 = s_0  \leq ... \leq s_K \}$.

\vspace{2mm}
For $s \leq t$, denoted $\DD^S_{s,t}$ the set \eqref{DS} corresponding to $\A=\F_s$, $\B=\F_t$, and to the minorant $m_{s,t}$ and majorant $M_{s,t}$.
Analogously for $\DD^{S,e}_{s,t}:= \DD^S_{s,t} \cap \DD^e$, cf. \eqref{DSe}.

\begin{proposition}
\label{propos1T}
Let us consider the time-consistent system $\big(x_{s,t} \big)_{s,t\in \T}$ on $(L_t)_{t\in\T}$ satisfying the sandwich condition \eqref{dynamic sand} with $(m_{s,t}, M_{s,t})_{s,t\in \T}$ fulfilling mM1.
For any $i<j$, consider the operators:
\begin{equation}
\label{rep:finite}
\hat x_{s_i, s_{j}}(X) := \esssup_{f \in {\cal Q}_{i,j}} \big\{ 
E[ f X|{\cal F}_{s_i}]  - \alpha_{s_i,s_{j}}(f) \big\} , \quad X\in L_p(\F_{s_j}),
\end{equation}
with the penalty
\begin{equation}\label{0}
\alpha_{s_i,s_j}(f) := \sum_{l=i}^{l=j-1} E[ \alpha_{s_l, s_{l+1}} ( g_{l+1}) \vert \F_{i}]
\end{equation}
where
$$
\alpha_{s_l,s_{l+1}}(g_{l+1}) := \esssup_{X \in L_{s_{l+1}}} \Big\{E[g_{l+1}X \vert \F_{s_l}] - x_{s_l, s_{l+1}}(X)\Big\}
$$
and
$$
{\cal Q}_{i,j} := \{f \in L_q(\F_{s_j})^+: \, f=\Pi_{l=i}^{j-1} g_{l+1}, \; g_{l+1}  \in \DD^{S,e}_{s_l, s_{l+1}} \}
$$ 
with $q=p(p-1)^{-1}$.
For all $s\leq t$ in $\T$, the operator $\hat x_{s,t}$ extends $x_{s,t}$ on $L_p(\F_t)$.
This family of operators $\big(\hat x_{s,t} \big)_{s,t\in \T}$ is a time-consistent sandwich preserving extension.
Moreover $\big(\hat x_{s,t} \big)_{s,t\in \T}$ is maximal, in the sense that, if $\big(\bar x_{s,t} \big)_{s,t\in \T}$ is another such family we have that: for all $i<j$,
$$
\hat x_{s_i,s_j}(X) \geq \bar x_{s_i,s_j}(X), \quad X \in L_p(\F_{s_j}).
$$
\end{proposition}

\vspace{2mm}
Note that from Corollary \ref{cor:S1},
$\alpha_{s_l,s_{l+1}} (g_{l+1})= \tilde x_{s_l,s_{l+1}} \big( E[ g_{l+1} \cdot \vert \F_{s_l} ]\big)$, where $\tilde x_{s_s,s_{l+1}}$ is the minimal penalty, see \eqref{pen-S2}.

\begin{proof}
From Theorem \ref{thmsand}, for every $i \le K-1$, we consider the maximal extension $\hat x_{s_i, s_{i+1}}$ of $x_{s_i,s_{i+1}}$. 
The operator $m_{0,T}$ is non degenerate. It follows from the weak time-consistency of $(m_{s,t})_{s,t \in \T}$ that for all $0 \leq s \leq t \leq T$, the operator $m_{s,t}$ is also non degenerate.
From Corollary  \ref{cor:S1b},  $\hat x_{s_i, s_{i+1}}$  admits a representation
\begin{equation}
\hat x_{s_i, s_{i+1}}(X)= \esssup_{g  \in \DD^{S,e}_{s_i,s_{i+1}}} \big\{E[ gX|{\cal F}_{s_i}] - \alpha_{s_i,s_{i+1}}(g)\big\}
\label{eq3.1}
\end{equation}
\[\begin{split}
\DD^{S,e}_{s_i,s_{i+1}} = &\{ g \in L_q({\cal F}_{s_{i+1}})^+ : \,E [g |{\cal F}_{s_i}] =1,\;g>0 \; P\;a.s. \\
&  m_{s_i,s_{i+1}}(X) \leq  E [gX |{\cal F}_{s_i}] \leq M_{s_i,s_{i+1}}(X), \, \forall X \in  L_p(\F_{s_{i+1}})^+  \}
\end{split}\]
and  
\begin{equation}\begin{split}
\alpha_{s_i,s_{i+1}}(g):= & \: \tilde x_{s_i, s_{i+1}}(E[g \cdot \vert \F_{s_i}])
 \\=\:
&\esssup_{Y \in  L_{s_{i+1}}} \big\{E[ gY|{\cal F}_{s_i}] -x_{s_is_{i+1}}(Y)\big\}.
\label{eq3.2}\end{split}
\end{equation}
For any $i<j$ define
\begin{equation*}
\hat x_{s_i, s_{j}}(X) := \esssup_{f \in {\cal Q}_{i,j}} \big\{ 
E[ fX|{\cal F}_{s_i}]  - \alpha_{s_i,s_{j}}(f) \Big\}
\end{equation*}
with the penalty
$$
\alpha_{s_i,s_j}(f) := \sum_{l=i}^{l=j-1} E[ \alpha_{s_l,s_{l+1}}(g_{l+1})|{\cal F}_{s_i}],
$$
for $f=g_{i+1}g_{i+2}\cdots g_{j}$
and
${\cal Q}_{i,j}$ 
as in the statement.
Note that for any $f \in {\cal Q}_{i,j}$ and any set $A\in \F_{s_i}$ we have $Q(A):= E[f1_A] = P(A)$.
We remark that the penalties $(\alpha_{i,j})_{i <j}$ satisfy the cocycle condition for the time instants in ${\cal T}$. 

The operator $x_{s_i, s_{i+1}}$ is weak $\F_{s_i}$-homogeneous, then $\alpha_{s_i, s_{i+1}}$ is local\footnote{i.e. for $f,g \in {\cal Q}_{s_i,s_{i+1}}$ and for $A \in {\cal F}_{s_i}$, the assertion $1_A E [fX|{\cal F}_{s_i}]=1_AE[gX|{\cal F}_{s_i}]\;\;\forall X \in L_{p}(\Omega,{\cal F}_{s_{i+1}},P)$ implies 
$
1_A \alpha_{s_i,s_{i+i}}(f)=1_A \alpha_{s_i,s_{i+1}}(g).
$
See \cite[Definition 4.1]{BN01}.}.
Observe that, for $A \in \mathcal{F}_{s_i}$,  $1_A E [f_1X|{\cal F}_{s_i}]=1_AE[f_2X|{\cal F}_{s_i}]\;\;\forall X \in L_{p}(\Omega,{\cal F}_{s_{i+1}},P)$ is equal to $1_Af_1 = 1_A f_2$.
Now we consider an argument by induction and we assume that $\alpha_{s_i,s_{j}}$ is local.
First of all recall that any element $\tilde f \in \mathcal{Q}_{s_i, s_{j+1}}$ can be of the form $fg$ where $f\in \mathcal{Q}_{s_i,s_j}$ and $g\in \mathcal{Q}_{s_j,s_{j+1}}$.
We consider $1_Af_1g_1 = 1_Af_2g_2$.
Then $E[1_Af_1g_1 \vert \F_{s_j}] = E[1_Af_2g_2 \vert \F_{s_j}]  $, which implies that
$1_A f_1 = 1_A f_2$. This in turns implies
\begin{equation}
\label{star 1}
1_A \alpha_{s_i,s_j}(f_1) = 1_A \alpha_{s_i,s_j}(f_2).
\end{equation}
Notice that $f_1>0\;P\;a.s.$. It follows that 
$
1_A g_1 = 1_Ag_2.
$
Hence
\begin{equation}\label{star 2}
1_A \alpha_{s_j, s_{j+1}} (g_1) = 1_A \alpha_{s_j, s_{j+1}} (g_2).
\end{equation} 
From \eqref{star 1} and \eqref{star 2} we conclude that:
$$
1_A \alpha_{s_i, s_{j+1}} (f_1g_1) = 1_A \alpha_{s_i, s_{j+1}} (f_2g_2)
$$
by the definition of $\alpha_{s_i, s_{j+1}}$ \eqref{0}.
Hence $\alpha_{s_i, s_{j+1}}$ is local as well.
The cocycle condition and the local property together imply the time-consistency of the system of operators  $\big(\hat x_{s,t} \big)_{s,t\in \T}$, see \cite[Theorem 4.4]{BN01}.

To conclude we show that the family $\big(\hat x_{s,t} \big)_{s,t\in \T}$ constitute a maximal extension. Indeed we have that, for all $i$, 
$$
\hat x_{s_i, s_{i+1}}(X) \geq \bar x_{s_i, s_{i+1}}(X), \quad X \in  L_p(\F_{s_{i+1}}).
$$
We proceed then by induction on $h$ such that $j=i+h$. Let $i<l<j$
\[\begin{split}
\hat x_{s_i, s_{j}}(X) =& \hat x_{s_i, s_{l}}(\hat x_{s_{l}, s_{j}}(X)) 
\geq \bar x_{s_i, s_{l}}(\hat x_{s_{l}, s_{j}}(X)) \\
&\geq  \bar x_{s_i, s_{l}}(\bar x_{s_{l}, s_{j}}(X)) 
= \bar x_{s_i, s_{j}}(X) , \quad X \in L_p(\B).
\end{split}\]
By this we end the proof.
\end{proof}

\begin{corollary}\label{corollary2}
For each $X \in L_p(\F_{s_j})$, there exists $f_X$ in $\mathcal{Q}_{i,j}$ such that
$$
\hat x_{s_i,s_j}(X) = E\big[ f_X X \vert \F_{s_i}\big] - \alpha_{s_i,s_j}(f_X).
$$
\end{corollary}
\begin{proof}
For $i=j-1$ apply Proposition \ref{propmax}:
$$
\hat x_{s_{j-1},s_j}(X) = E\big[ f_{X,j}X \vert \F_{s_{j-1}}\big] - \alpha_{s_{j-1}, s_j}(f_{X,j}).
$$
From Lemma \ref{lemmaequiv0}, $ f_{X,j}$ belongs to $\DD^{S,e}_{s_i,s_{i+1}}$.
From the time-consistency of $\big(\hat x_{s,t} \big)_{s,t\in \T}$ and the definition of $\alpha_{s_i,s_j}$ in \eqref{0} we have
$$
f_X = \prod_{l=i}^{j-1} f_{X, l+1}.
$$
By this we end the proof.
\end{proof}

\bigskip
\subsection{Countable discrete time systems}
\label{secdiscrete}

Let us now consider a countable set $\T\subset [0,T]$, with $0,T \in \T$, and a sequence of finite sets $(\T_n)_{n=1}^\infty$:  $\T_n \subseteq \T_{n+1}$, such that $\T = \cup_{n=1}^\infty \T_n$.
Let us consider the time-consistent system $\big(x_{s,t} \big)_{s,t\in \T}$ on $(L_t)_{t\in\T}$ satisfying the sandwich condition \eqref{dynamic sand} with $(m_{s,t}, M_{s,t})_{s,t\in \T}$ fulfilling mM1.
\begin{lemma}
For any $n$, let $( x^n_{s,t})_{s,t \in \T_n}$ be the maximal time-consistent sandwich preserving extensions of $(x_{s,t})_{s,t \in \T_n}$. 
Now consider $s,t, \in {\cal T}$. Let $n_0\in \mathbb{N}$ such that $s,t \in \T_{n_0}$. 
Then, for any $n > n_0$ and $X \in L_p({\cal F}_t)$, the sequence $( x^n_{s,t}(X))_{n>n_0}$ is non increasing $P\;a.s.$ 
Hence it admits a limit 
\begin{equation}
\label{e1}
\hat x_{s,t}(X):= \lim_{n\to \infty}x^n_{s,t}(X).
\end{equation}
Moreover, for $n>n_0$, let $\alpha^n_{s,t}$ be the  minimal penalty associated to  $ x^n_{s,t}$. 
This penalty has representation
\begin{equation}
\alpha^n_{s,t}(Q):=\esssup_{X \in L_p({\cal F}_s)}(E_Q(X|{\cal F}_s)- x^n_{s,t}(X)),
\label{eqn}
\end{equation}
for all probabilility measure $Q \sim P$, where $\alpha^n_{s,t}(Q) = \alpha^n_{s,t}(f)$ with $f=\frac{dQ}{dP}$.
Then, for all $Q \sim P$, the sequence $(\alpha^n_{s,t}(Q))_{n>n_0}$ is non negative and non decreasing $P\;a.s.$.
Hence it admits a limit 
\begin{equation}
\label{e2}
\hat \alpha_{s,t}(Q):= \lim_{n\to\infty} \alpha^n_{s,t}(Q).
\end{equation}

\label{lemmahatx}
\end{lemma}

\begin{proof}
The extensions $(x^n_{s,t})_{s,t \in \T_n}$ are maximal over all other sandwich preserving extensions time-consistent on $\T_n$.  Then, for $s,t, \in {\cal T}$ and $n>n_0$, we can regard the extension $x^{n+1}_{s,t}$  as another sandwich preserving extension of $x_{s,t}$, $(x^{n+1}_{s,t})_{s,t\in \T_n}$ is time-consistent on $\T_n$. Thus $ x^n_{s,t}(X) \geq x^{n+1}_{s,t}(X)$. 

From Corollary \ref{corollary2}, $ x^n_{s,t}$ admits a representation with equivalent probability measures. The result for $\alpha^n_{s,t}(Q)$, $Q \sim P$, is then an immediate consequence of equation (\ref{eqn}).
\end{proof}

\begin{theorem}
\label{teo1T}
Let us consider the discrete time-consistent system $\big(x_{s,t} \big)_{s,t\in \T}$  on $(L_t)_{t\in\T}$ satisfying the sandwich condition \eqref{dynamic sand} with mM1. 
Then each operator in this family admits an extension to the whole $L_p(\F_t)$ with values in $L_p(\F_s)$ satisfying the sandwich condition and such that the family of extensions is time-consistent.
In particular, the family of operators $(\hat x_{s,t})_{s,t\in \T}$
given in Lemma \ref{lemmahatx} is a time-consistent and sandwich preserving extension of $\big(x_{s,t} \big)_{s,t\in \T}$.
Moreover, for any $s\leq t$, the operators $\hat x_{s,t}$ \eqref{e1} and  $\hat \alpha_{s,t}$ \eqref{e2} satisfy the relationship:
\begin{align}
\label{count-st}
\hat x_{s, t}(X) &= \esssup_{Q \sim P} ( E_Q[ X|{\cal F}_{s}]  - \hat \alpha_{s,t}(Q)) \nonumber\\
&= \esssup_{f \in \DD^{S,e}_{s,t}} ( E[ f X|{\cal F}_{s}]  - \hat \alpha_{s,t}(f)), \quad X\in L_p(\F_t).
\end{align}
Moreover, for all $X$ there is $f_X \in \DD_{s,t}^{S,e}$ such that 
\begin{equation}
\label{fx}
\hat x_{s,t}(X)=E(f_XX|{\cal F}_s)-\hat \alpha_{s,t}(f_X).
\end{equation}
This extension is maximal, in the sense that, for any other such extension $\big(\bar x_{s,t} \big)_{s,t\in \T}$ we have that: for all $s<t \in \T$,
$$
\hat x_{s,t}(X) \geq \bar x_{s,t}(X), \quad X \in L_p(\F_t).
$$
Also for all $s,t \in {\cal T}$, $\hat \alpha_{s,t}$ is the minimal penalty associated to $\hat x_{s,t}$.
\end{theorem}

\begin{proof}
In Lemma \ref{lemmahatx} we have defined, for all $s,t \in \T$, 
$$
\hat x_{s,t}(X):= \lim_{k\to\infty} x^k_{s,t}(X), \qquad X\in L_p(\F_t),
$$ 
where $x_{s,t}^k$ is the maximal extension of $x_{s,t}$ on $\T_k$ and for $f \in \DD^{S,e}_{s,t}$ with $s,t\in \T_k$, we have set $\hat \alpha_{s,t}(f):=\hat \alpha_{s,t}(Q)$,  $\alpha^k_{s,t}(f)=\alpha^k_{s,t}(Q)$, where $f= dQ/dP$, with
$$
\hat \alpha_{s,t}(f):= \lim_{k\to\infty} \alpha^k_{s,t}(f), \qquad f \in \DD^{S,e}_{s,t}
$$ 

{\it Step 1: Proof of the representations \eqref{count-st}, \eqref{fx} and the sandwich property.}

Let $s, t  \in \T$. 
Fix $X  \in L_p({\cal F}_t)$. 
For every $k$, from Corollary \ref{corollary2}, there is $f_{X,k} \in \DD^{S,e}_{s,t}$ such that 
\begin{equation}
\label{t3.3-1}
 x^k_{s,t}(X)=E(f_{X,k} X |{\cal F}_s)- \alpha^k_{s,t}(f_{X,k}).
 \end{equation}

From Lemma \ref{lem.tech} the set $\DD^S_{s,t}$ is compact for the weak* topology, thus there is a subsequence of $(f_{X,k})_k$ converging to $f_X \in \DD_{s,t}^S$. 
Without loss of generality we can assume that the sequence $(f_{X,k})_k$ itself has the limit $f_X$ (for the weak* topology). 
Fix $n>n_0$. 
From equation \eqref{eqn}, $\alpha^n_{s,t}$ is lower semi continuous for the weak* topology thus 
 $$  
 \alpha^n_{s,t}(f_X) \leq \liminf_{k \rightarrow \infty} \alpha^n_{s,t} (f_{X,k}).
 $$
From Lemma \ref{lemmaequiv0}  it follows that $f_X \in \DD^{S,e}_{s,t}$. 
From Lemma \ref{lemmahatx}, for  given $k$, the sequence $(\alpha^n_{s,t}(f_{X,k}))_n$ is non decreasing. 
Therefore for every $k \geq n$, $\alpha^n_{s,t} (f_{X,k})\leq  \alpha^k_{s,t} (f_{X,k})$. 
Thus by \eqref{t3.3-1},
 $$ 
  \alpha^n_{s,t}(f_X) \leq \liminf_{k \rightarrow \infty}\Big(E(f_{X,k} X|{\cal F}_s))- x^k_{s,t}(X)\Big).
  $$ 
  Passing to the limit as $k \rightarrow \infty$, we get the inequality 
$$ 
\alpha^n_{s,t}(f_X) \leq E(f_X X|{\cal F}_s))-\hat x_{s,t}(X).
$$
Letting  $n \rightarrow \infty$
 \begin{equation}
 \hat\alpha_{s,t}(f_X) \leq E(f_X X|{\cal F}_s))-\hat x_{s,t}(X).
\label{eqalphahat}
\end{equation}
On the other hand, for every $Q \sim P$,  and $Y \in L_p({\cal F}_t)$, for every $n$,
$  \alpha^n_{s,t}(Q) \geq E_Q(Y|{\cal F}_s))- x^n_{s,t}(Y)$.
Passing to the limit this gives:
\begin{equation}
 \hat \alpha_{s,t}(Q) \geq E_Q(Y|{\cal F}_s))-\hat x_{s,t}(Y).
\label{eqalphag2}
\end{equation}
It follows that 
$$
\hat x_{s,t}(X)=E(f_X X|{\cal F}_s))-\hat \alpha_{s,t}(f_X).
$$ 
Then from the above equation and (\ref{eqalphag2}) we have proved the representations \eqref{count-st} and \eqref{fx}.

Notice that the sandwich condition follows from the sandwich condition for $x^n_{s,t}$ passing to the limit for $n\to\infty$, see \eqref{e1}.

\vspace{1mm}
{\it Step 2: Time-consistency.}

From \eqref{count-st}, $\hat x_{s,t}$ is lower semi continuous. From the definition of $\hat x_{s,t}$ as the limit of $x^n_{s,t}$ it follows that $\hat x_{s,t}$ extends  $x_{s,t}$ for every $s, t \in {\cal T}$.
On the other hand for every $r \leq s \leq t$ in $\T$ and every $n$ large enough such that $r,s,t$ belong to $\T_n$, we already know that $(x^n_{s,t})_{s,t\in \T}$ is $\T_n$ time-consistent. 
We recall that the minimal penalty of a time-consistent family of operators satisfies the local property \cite[Lemma 2.3]{BN01} and the cocycle condition \cite[Theorem 2.5]{BN01}.
The family of penalties $(\alpha^n_{s,t})_{s,t\in\T}$ satisfies the cocycle condition: 
$$
 \alpha^n_{r,t}(Q)= \alpha^n_{r,s}(Q)+E_Q( \alpha^n_{s,t}(Q) |{\cal F}_r), \quad  Q \sim P.
 $$
Hence, passing to the limit for the non decreasing sequence $(\alpha^n_{s,t})_{s,t\in\T}$ we get the cocycle condition:
 $$
\hat \alpha_{r,t}(Q)= \hat \alpha_{r,s}(Q)+E_Q( \hat \alpha_{s,t}(Q)  |{\cal F}_r), \quad  Q \sim P.
 $$
The local property is obtained in the same way.
And thus from \cite[Theorem 4.4]{BN01}, $(\hat x _{s,t})_{s,t\in\T}$ is time-consistent. 

\vspace{1mm}
{\it Step 3: Maximality of $\hat x_{s,t}$.}

Notice that if another family $\overline x_{s,t}$  satisfies all the above properties. Necessarily for all $s,t \in \T$ and  $n$ large enough such that $s,t \in \T_n$, 
from the maximal property of $ x^n_{s,t}$ it follows that 
$\overline x_{s,t}(X) \leq  x^n_{s,t}(X)$. Thus passing to the limit we get the maximality for $\hat x_{s,t}$.

\vspace{1mm}
{\it Step 4: Minimality of $\hat \alpha_{s,t}$.}

To see that $\hat \alpha_{s,t}$  is the minimal penalty associated to $\hat x_{s,t}$, we proceed as follows.
From Lemma  \ref{lemmahatx}, the sequence $( x^n_{s,t}(X))_{n>n_0}$ is non increasing $P\;a.s.$.\\
We already know that  $ \alpha^n_{s,t}$  is the minimal penalty associated to  $( x^n_{s,t}(X))$.Thus 
\begin{align*}
\alpha^n_{st}(Q)&=\esssup_{X \in L_p(\F_t)} \big(E_Q (X|{\cal F}_s) - x^n_{s,t}(X) \big) \\
& \leq \esssup_{X\in L_p(\F_t)} \big( E_Q  (X|{\cal F}_s) - \hat x_{s,t}(X) \big) .
\end{align*}
Passing to the limit, we have 
$$
\hat \alpha_{st}(Q)\leq \esssup_{X \in L_p(\F_t)} \big( E_Q (X|{\cal F}_s) - \hat x_{s,t}(X) \big).
$$
On the other hand, for all $X$ we have
$$
\alpha^n_{st}(Q) \geq  E_Q(X|{\cal F}_s) - x^n_{s,t}(X).
$$
Passing to the limit, we have 
$$\hat \alpha_{st}(Q) \geq E_Q(X|{\cal F}_s) - \hat x_{s,t}(X),\;\; \forall {X \in L_p(\F_t)}. $$
Hence 
$$\hat \alpha_{st}(Q)= \esssup_{X \in L_p(\F_t)} \big(E_Q(X|{\cal F}_s) - \hat x_{s,t}(X)\big) .$$
\end{proof}


\section{Sandwich extensions of continuous time systems of operators}

In this section we study sandwich preserving extensions for a system of operators $(x_{s,t})_{s,t \in [0,T]}$.
These extensions are time-consistent.
We stress that to obtain a time-consistent extension it is not enough to collect all the extensions of single operators in one family.
Time-consistency is achieved with some careful procedure of extension involving the representation of the operators and an appropriate passage from discrete to continuous time.
For this we first define the system of majorant and minorant operators serving as bounds in the sandwiches. 

\begin{definition}
We say that the family $(m_{s,t},M_{s,t})_{s,t\in [0,T]}$ satisfies the mM2-condition if 
\begin{enumerate}
\item mM1 is satisfied (Definition \ref{mM1});
\item  $\esssup_{s \leq T}(M_{s,T}(X))$  belongs to $L_p({\cal F}_T)^+$ for all $X \in L_p({\cal F}_T)^+$;
\item  for every $X \in L_p({\cal F}_t)^+$,
\begin{equation}
m_{s,t}(X)=lim_{t'>t,t'\downarrow t} m_{st'}(X);
\label{wtcm2}
\end{equation}
\label{defwtc}
\item for every $X \in L_p({\cal F}_t)^+$,
\begin{equation}
M_{s,t}(X)=lim_{t'>t,t'\downarrow t} M_{st'}(X);
\label{eqwtcM2}
\end{equation}
\item  for every $X \in L_p({\cal F}_t)^+$,
\begin{equation}
m_{s,t}(X) \leq \limsup_{s'>s,s'\downarrow s} m_{s't}(X);   \;\;\;\;\;       M_{s,t}(X) \geq \liminf_{s'>s,s'\downarrow s} M_{s't}(X);
\label{wrc}
\end{equation} 
 \end{enumerate}
\label{defmM2}
\end{definition}

Let ${\cal T}$ be a countable dense subset of $[0,T]$ containing $0$ and $T$. 
\begin{definition} A system $\big(x_{s,t}\big)_{s,t \in [0,T]}$ on $(L_t)_{t \in [0,T]}$  is right-continuous if for all $t$, all $X \in L_t$, and all sequences $(s_n)_n$, $s<s_n \leq t$, $s_n \downarrow s$, $x_{s,t}(X)=\lim_{n \rightarrow \infty}x_{s_n,t}(X)$, where the convergence is $P$  a.s.
\end{definition}

\begin{lemma} 
\label{lemmaQ}
Assume mM2 condition. Let $\big(\hat x_{s,t} \big)_{s,t\in \T}$ and $\big(\hat \alpha_{s,t} \big)_{s,t\in \T}$ be as in Lemma  \ref{lemmahatx}.  There is a probability measure $Q_0$ equivalent to $P$ such that for all $s,t \in {\cal T}$, $0 \leq s \leq t \leq T$,  $\hat \alpha_{s,t}(Q_0)=0$.
\end{lemma}
\begin{proof}
From Theorem \ref{teo1T}, there is a probability measure $Q_0$ such that $0=\hat x_{0,T}(0)=-\hat \alpha_{0,T}(Q_0)$. 
From Lemma \ref{lemmaequiv0}, $Q_0$ is equivalent to $P$. It follows  from the ${\cal T}$-cocycle condition and the non negativity of the penalty that  $\hat \alpha_{s,t}(Q_0)=0$ for all $s \leq t$ in ${\cal T}$.
\end{proof}

\begin{proposition}
The notations are those of Lemma \ref{lemmaQ}. 
\begin{enumerate}
\item For all $X\in L_p({\cal F}_T)$, $\big(\hat x_{s,T}(X)\big)_{s\in \T}$ is a $Q_0$-supermartingale.
\item For every sequence $(s_n)_n $ in ${\cal T}$ decreasing to $s \in {\cal T}$, 
$E_{Q_0}(\hat x_{s_n,T}(X))$ has the limit $E_{Q_0}(\hat x_{s,T}(X))$,  for all $X\in L_p({\cal F}_T)$.
\end{enumerate}
\label{propext2}
\end{proposition}
\begin{proof}
\begin{enumerate}
\item It follows easily from Lemma \ref{lemmaQ}, the ${\cal T}$ time-consistency, and the representation of $\hat x_{s,T}(X)$  \eqref{count-st}.
\item The proof is inspired by the one  of Lemma 4 in \cite{BN02}. The main differences are the facts that here $\hat x_{s,t}$ is only defined for $s,t$ in ${\cal T}$ and that the   operator is  defined on $L \subset L_p({\cal F}_t)$ $1 \leq p \leq \infty$,  while  in \cite{BN02} the dynamic risk measure $\rho_{s,t}$ was defined on $L_{\infty}({\cal F}_t)$ and time-consistency was  considered for all real indexes. \\ Let $X \in L_p({\cal F}_T)$ and $s \in {\cal T}$.  From Theorem \ref{teo1T},  there is an $f_X \in \DD^{S,e}_{s,T}$ such that 
$\hat x_{s,T}(X)=E(f_XX|{\cal F}_s)-\hat \alpha_{s,T}(f_X)$.  Let $R_X$ be the probability measure such that $\frac{dR_X}{dP}=f_X$. It follows from the cocycle condition that 
\begin{eqnarray}
\label{eq4.1}
\hat x_{s,T}(X)=E_{R_X}(X|{\cal F}_s)-\hat \alpha_{s,T}(R_X)\nonumber \\=E_{R_X}[E_{R_X}(X|{\cal F}_{s_n})-\hat \alpha_{s_n,T}(R_X)]|{\cal F}_s)-\hat \alpha_{ss_n}(R_X) .
\end{eqnarray}
Furthermore $E_{R_X}(X|{\cal F}_{s_n})-\hat \alpha_{s_n,T}(R_X) \leq \hat x_{s_n,T}(X)$ and the penalties are non negative. 
Hence, we have that 
\begin{equation}
\label{eq4.1b}
\hat x_{s,T}(X) \leq  E_{R_X}(\hat x_{s_n,T}(X)|{\cal F}_s)
\end{equation}
and
\begin{eqnarray}
 \label{eq4.2}
\hat x_{s,T}(X) \leq  E_{R_X}(\hat x_{s_n,T}(X)|{\cal F}_s)&=&E(f_X\hat x_{s_n,T}(X)|{\cal F}_s)\nonumber\\
&=&E(E(f_X|{\cal F}_{s_n})\hat x_{s_n,T}(X)|{\cal F}_s)  . \nonumber
\end{eqnarray}
Let $g \in L_q({\cal F}_s)$ be the Radon Nykodym derivative of the restriction of $Q_0$ to ${\cal F}_s$.  
Taking the $Q_0$ expectation, $g$ being ${\cal F}_s$-measurable, we obtain
\begin{eqnarray}
E_{Q_0}(\hat x_{s,T}(X)) &\leq& E[gE(f_X|{\cal F}_{s_n})\hat x_{s_n,T}(X)]   \label{eq4.2.0} \\
&=  & E(g (\hat x_{s_n,T}(X))
+E[\hat x_{s_n,T}(X)(g(E(f_X|{\cal F}_{s_n})-1))]   \nonumber
\end{eqnarray}
The density $f_X$ belongs to $\DD^S_{s,T}$ and $g$ belongs to $\DD^S_{0,s}$, so $gf_X$ belongs to $L_q({\cal F}_T)$ and  $g(E(f_X|{\cal F}_{s_n})-1)$ has limit $0$ in $L_q({\cal F}_T)$.  \\
For all $X$, $|\hat x_{s_n T}(X)| \leq \hat x_{s_n,T}(|X|) \leq M_{s_n,T}(|X|)$. From  property {\it 2.} in Definition  \ref{defmM2}, 
$\sup_n E((|\hat x_{s_n,T}(X)|^p)<\infty$. It follows from H\"older inequality that $ \epsilon_n(X):=E[\hat x_{s_n,T}(X)(g(E(f_X|{\cal F}_{s_n})-1))]$ has limit $0$. 
Similarly, 
\begin{eqnarray}
\delta_n(X)&:=&-E_{Q_0}(\hat x_{s_n,T}(X))+E(g\hat x_{s_n,T}(X))\nonumber\\
&=&-E\big{( [}E(\frac{dQ_0}{dP}|{\cal F}_{s_n})+E(\frac{dQ_0}{dP}|{\cal F}_{s})\big{]}\hat x_{s_n,T}(X)\big{)}\nonumber
\end{eqnarray}
has limit $0$.
Observe that, from the $Q_0$-supermartingale property (point {\it 1.}), if follows that 
\begin{equation*}
E_{Q_0}(\hat x_{s_n,T}(X)) \leq E_{Q_0}(\hat x_{s,T}(X)) .
\end{equation*}
Then from \eqref{eq4.2.0} we obtain that
$$ E_{Q_0}(\hat x_{s,T}(X))-\epsilon_n(X)-\delta_n(X) \leq E_{Q_0}(\hat x_{s_n,T}(X)) \leq E_{Q_0}(\hat x_{s,T}(X)).$$
This proves the result.
\end{enumerate}
\end{proof}

\begin{theorem}
\label{teo2T2}
Let us consider a right-continuous time-consistent system of operators $\big(x_{s,t} \big)_{s,t\in [0,T]}$ of type (2.1) defined on $(L_t)_{t\in [0,T]}$ satisfying the sandwich condition with  mM2.\\ 
Then there is  a right-continuous, time-consistent, sandwich preserving extension  $\big(\hat x_{s,t} \big)_{s,t\in [0,T]}$ defined on the whole $(L_p(\F_t))_{t \in [0,T]}$.  
One such extension can be represented as
\begin{equation}
\label{extension in continuous}
\hat x_{s,t}(X)=\esssup_{R \in {\cal R}} [E_R(X|{\cal F}_s) -\hat\alpha_{s,t}(R)], \quad X \in L_p(\mathcal{F}_t),
\end{equation}
with 
\begin{equation}
{\cal R} :=\{R \sim P:\; \hat \alpha_{0,T}(R)<\infty\}
\label{eqR}
\end{equation}
and $\hat \alpha_{0,T}$ is the minimal penalty associated to $\hat x_{0,T}$ as in Theorem \ref{teo1T}.
Also for any $X \in L_p(\mathcal{F}_t)$, there exists $R_X \in \mathcal{R}$ such that
$$
\hat x_{s,t}(X)=E_{R_X}(X|{\cal F}_s)-\hat \alpha_{s,t}(R_X)\quad \forall s \leq t.
$$
Furthermore for all $t>0$, and all $X \in L_p({\cal F}_t)$, $\hat x_{s,t}(X)_{0 \leq s \leq t}$ admits a c\`adl\`ag version.
\label{thmextcontinuous2}
\end{theorem}

\begin{proof}
\vspace{1mm}
{\it Step 1:  Definition of the extension $\hat x_{s,t}$ for indices in the whole $[0,T]$.}

Let ${\cal T}$  be a countable dense subset of $[0,T]$ containing $0$ and $T$. Let $\big(\hat x_{s,t} \big)_{s,t\in {\cal T}}$ be the time-consistent extension of  $\big(x_{s,t}\big)_{s,t \in {\cal T}}$ constructed in Section \ref{secdiscrete}. 
Let $X \in L_p({\cal F}_t)$.  
From Proposition \ref{propext2} and from the Modification Theorem (Chapter VI Section 1 in \cite{DM}) applied to the $Q_0$-supermartingale $(\hat x_{s,T}(X))_{s \in {\cal T}}$, it follows that $(\hat x_{s,T}(X))_{s \in {\cal T}}$  admits a modification which can be extended into a c\`adl\`ag process $(\hat x_{s,T}(X))_{s \in [0,T]}$ defined for all $s \in [0,T]$. \\
Notice that from Remark \ref{4.1},  $x_{s,t}$ coincides on $L_t$ with the  restriction of $x_{s,T}$. 
For $0 \leq s\leq t \leq T$, we define  $\hat x_{s,t}$ to be the restriction of $\hat x_{s,T}$  on $L_p({\cal F}_t)$. 
It follows that $\hat x_{s,t}$ is an extension of $x_{s,t}$ to the whole $L_p({\cal F}_t)$. If $s,t \in {\cal T}$ this extension   coincides also with the construction of $\hat x_ {st}$ given in the previous section. To see this, consider the ${\cal T}$ time-consistency of $\hat x_ {st}$
(of the previous section) and its projection property.

\vspace{1mm}
{\it Step 2: Extension of the penalty on a set ${\cal R}$ of probability measures and right-continuity.} 

Consider the penalties $(\hat \alpha_{s,t})_{s,t\in \T}$ associated to $(\hat x_{s,t})_{s,t\in \T}$ given in Section \ref{secdiscrete}.
Define
\begin{equation*}
{\cal R} :=\{R \sim P:\; \hat \alpha_{0,T}(R)<\infty\}
\label{eqR}
\end{equation*}

 Let $X$ in $L_p({\cal F}_T)$. From Theorem \ref{teo1T} (see \eqref{fx}),   there is a probability measure $R_X \sim P$ such that
\begin{equation}
\label{eqrep0}
\hat x_{0,T}(X)=E_{R_X}(X)-\hat \alpha_{0,T}(R_X).
\end{equation}
Thus ${\cal R}$ is non empty.

Let $s\in [0,T]$ and consider $(s_n)_n \subset {\cal T}$, $s_n \downarrow s$.  For every probability measure $R \in {\cal R}$, the sequence $E_R(\hat \alpha_{s_n,T}(R)|{\cal F}_s)$ is increasing. Thus it admits a limit.  
From Section  \ref{secdiscrete} we can see that, if $s$ belongs to ${\cal T}$, this limit is equal to  $\hat \alpha_{sT}(R)$. 
Indeed it is enough to exploit the representation as minimal penalty. Then by right-continuity of the filtration and the c\`adl\`ag extension of Step 1 we have
\[
\begin{split}
\hat \alpha_{s,T} (R) \geq& 
\lim_{n\to\infty} E_R [\hat \alpha_{s_n,T}(R) \vert \mathcal{F}_s]\\
\geq 
&E_R [ \esssup_{X \in L_p(\mathcal{F}_T) } \lim_{n\to\infty} ( E_R [X \vert \mathcal{F}_{s_n} ] - \hat x_{s_n,T}(R) \vert \mathcal{F}_s ] = \hat \alpha_{s,T} (R).
\end{split}
\]
If $s\notin {\cal T}$ we can define $\hat \alpha_{sT}(R)$ as the limit of  $E_R(\hat \alpha_{s_n,T}(R)|{\cal F}_s)<\infty$  $R \;a.s.$.
Moreover, for $r,s \in \T$, due to the ${\cal T}$ time-consistency proved in  Section \ref{secdiscrete}, $\hat \alpha_{r,s}(R)$ satisfies the cocycle condition on $\T$. 
Note that $\sup_{s \leq T} E_R(\hat \alpha_{sT}(R)) =\hat \alpha_{0T}(R) < \infty$.
Then we can  define for all $0 \leq r \leq s \leq T$
\begin{equation}
 \hat \alpha_{r,s}(R):=\hat \alpha_{rT}(R)-E_R(\hat \alpha_{s,T}(R)|{\cal F}_r)
\label{eqdefpen}
\end{equation} 
Thus  $\hat \alpha_{r,s}(R)$ is now defined for all indices in $[0,T]$ and all $R \in {\cal R}$. Moreover $\hat \alpha_{r,s}(R)$ satisfies the cocycle condition. It follows also from the right-continuity of $\hat \alpha_{s,T}(R)$ in $s$  and from the cocycle condition that for all $R \in {\cal R}$, $\hat \alpha_{r,s}(R)$ is also   right-continuous in $s$.


\vspace{1mm}
{\it Step 3: Representation of $\hat x_{r,t}(X)$.}

 Let $0 \leq r  \leq t \leq T$. Let $X \in L_p({\cal F}_t)$.  Let $R_X \in {\cal R}$ such that (\ref{eqrep0}) is satisfied. Now we prove that, for all $r \leq t$,
\begin{equation}
\hat x_{r,t}(X)=E_{R_X}(X|{\cal F}_r)-\hat \alpha_{r,t}(R_X).
\label{eqrepR}
\end{equation}
 In fact, 
making use of the cocycle condition and then of the ${\cal T}$ time-consistency, we get that for all $r \in {\cal T}$, 
\begin{eqnarray*}
\hat x_{0,T}(X)=E_ {R_X}(E_ {R_X}(X|{\cal F}_r)-\hat \alpha_{r,T}( {R_X}))-\hat  \alpha_{0,r}( {R_X}) \nonumber\\
\leq E_ {R_X}(\hat x_{r,T}(X) ) -\hat \alpha_{0,r}( {R_X}) \leq \hat x_{0,r}(\hat x_{r,T}(X))=\hat x_{0,T}(X).
\label{eqnhat}
\end{eqnarray*}
Thus every inequality in the expression above is an equality.  
In particular 
\begin{equation}
\hat x_{r,T}(X)=E_{R_X}(X|{\cal F}_r)-\hat \alpha_{r,T}(R_X), \;\;R_X \; a.s.\; \forall r\in {\cal T},
\label{eqhatalpha1}
\end{equation}
and
\begin{equation}
\hat x_{0,r}(\hat x_{r,T}(X))=E_R(\hat x_{r,T}(X))-\hat \alpha_{0,r}(R_X),\;\;R_X \; a.s. \;\forall r \in {\cal T}.
\label{eqhatalpha2}
\end{equation}
Consider a sequence $(r_n)_n \subset {\cal T}$, $r_n \downarrow r$, with $r \in [0,T]$.
Passing to the limit in the corresponding equation to (\ref{eqhatalpha1}), we can see that
\begin{equation}
\hat x_{r,T}(X)=E_{R_X}(X|{\cal F}_r)-\hat \alpha_{r,T}(R_X), \;\;R_X \; a.s.\; \forall r\in [0,T].
\label{eqhatalpha1g}
\end{equation}

We have assumed that $X \in L_p({\cal F}_t)$ thus  $\hat x_{t,T}(X)=X$.  It follows then from (\ref{eqhatalpha1g}) applied with $r=t$, that  $\hat \alpha_{t,T}(R_X)=0$.

Then, again from  (\ref{eqhatalpha1g}) and the cocycle condition for $R_X \in {\cal R}$, it follows that, for all $X \in  L_p({\cal F}_t)$,
\begin{equation}
\hat x_{r,t}(X)=\hat x_{r,T}(X)=E_{R_X}(X|{\cal F}_r)-\hat \alpha_{r,t}(R_X)\;R_X \; a.s.\; \forall r,t\in [0,T].
\label{eqhatalpha5}
\end{equation}

\vspace{1mm}
{\it Step 4: Another  representation of $\hat x_{r,t}$ for all $0 \leq r \leq t \leq T$.}

Let $X \in L_p({\cal F}_t)$. 
We will prove that  $\hat x_{r,t}(X)$ admits the following representation:
\begin{equation}\begin{split}
\hat x_{r,t}(X)= & \esssup_{R \in {\cal R}} [E_R(X|{\cal F}_r) -\hat\alpha_{r,t}(R)] \\
= &E_{R_X}(X|{\cal F}_r)-\hat \alpha_{r,t}(R_X),
\end{split}
\label{eqrs}
\end{equation}
where $R_X$ satisfies equation (\ref{eqrep0}).
For all $r,t$ in ${\cal T}$, being $\hat \alpha_{r,t}$ the minimal penalty, we have that for all $R \in {\cal R}$, 
\begin{equation}
\hat  \alpha_{r,t}(R) \geq E_{R}(X|{\cal F}_r)-\hat x_{r,t}(X), \;  \forall X \in L_p({\cal F}_t).
\label{eqhatalpha6}
\end{equation}
The above equation can also be written
\begin{equation}
\hat  \alpha_{r,t}(R) \geq E_{R}(X|{\cal F}_r)-\hat x_{r,T}(X), \;  \forall X \in L_p({\cal F}_t).
\label{eqhatalpha7}
\end{equation}

Exploiting the right-continuity of the filtration and of $\hat \alpha_{r,t}$ and $\hat x_{r,t}$ as discussed in Steps 1 and 2, we can see that, passing to the limit in $r$ and then to the limit in $t$, equation (\ref{eqhatalpha7}) and thus also  (\ref{eqhatalpha6}) are satisfied for all $r \leq t$ in $[0,T]$.
 Equation (\ref{eqrs}) follows then from (\ref{eqhatalpha5}) and (\ref{eqhatalpha6}).

\vspace{1mm}
{\it Step 5 : Time consistency of $\hat x_{s,t}$.}

Set $0 \leq r \leq s \leq t \leq T$.
Let $ X \in L_p({\cal F}_t)$. Let $R_X \in {\cal R}$ such that  (\ref{eqrep0}) is satisfied.
From (\ref{eqrs}), for all $Y \in L_p({\cal F}_s)$, 
$\hat x_{r,s}(Y) \geq E_{R_X}(Y|{\cal F}_r) -\hat\alpha_{r,s}(R_X)$.
With $Y=\hat x_{s,t}(X)$, making use of equation (\ref{eqhatalpha5}) and of the cocycle condition,  it follows  that 
\begin{equation}
\hat x_{r,s}(\hat x_{s,t}(X)) \geq \hat x_{r,t}(X), \;\; \forall r \leq s \leq t.
\label{eqtc51}
\end{equation}
On the other hand, let $R_Y \in  {\cal R}$ such that  
\begin{equation}
\label{eqrep0Y}
\hat x_{0,T}(Y)=E_{R_Y}(Y)-\alpha_{0,T}(R_Y)
\end{equation}
(cf. \eqref{eqrep0}).
From equation (\ref{eqhatalpha5}), it follows that  for all $r, s$: $r \leq s$,
\begin{equation}
\hat x_{r,s}(Y)=E_{R_Y}(Y|{\cal F}_r)-\hat \alpha_{r,s}(R_Y)\;R_Y \; a.s.
\label{eqhatalpha5Y}
\end{equation}
We already know that $(\hat x_{s,t})_ {s,t \in {\cal T}}$ satisfies ${\cal T}$ time consistency, and the sandwich condition. 
From the right-continuity of $(\hat x_{s,t}(X))_{s \in [0,t]}$ for all $t \geq 0$ (see Step 1) and the hypothesis mM2, item 3,  it follows that  $\esssup_{s \leq t}|\hat x_{s,t}(X)| $ belongs to $L_p({\cal F}_t)^+$.  Let $s_n \in {\cal T}$ such that $s_n \downarrow s$.  From the right-continuity and the dominated convergence theorem, it follows that $E_{R_Y}(\hat x_{s,t}(X)|{\cal F}_r)=\lim_{n \rightarrow \infty}E_{R_Y} (\hat x_{s_n,t}(X)|{\cal F}_r)$. 
 The probability measure $R_Y$ belongs to ${\cal R}$ thus from step 2, $\hat \alpha_{r,s}(R_Y)=\lim_{n \rightarrow \infty} \hat \alpha_{r,s_n}(R_Y) $.  From equation (\ref{eqhatalpha5Y}), we then get that
\begin{equation}
\hat x_{r,s}(\hat x_{s,t}(X))=\hat x_{r,s}(Y)= \lim _{n \rightarrow \infty} [E_{R_Y}(\hat x_{s_n,t}(X)|{\cal F}_r)-\hat \alpha_{r,s_n}(R_Y)].
\label{eqhatalpha5Yb}
\end{equation}
From equation (\ref{eqrs}), for all $n$,  
$$
E_{R_Y}(\hat x_{s_n,t}(X)|{\cal F}_r)-\hat \alpha_{r,s_n}(R_Y) \leq \hat x_{r,s_n}(\hat x_{s_n,t}(X)).
$$
 In the case $r$ belongs to ${\cal T}$, applying the ${\cal T}$time consistency, we have that the right-hand side of the above equation is equal to $\hat x_{r,t}(X)$.  Thus
\begin{equation}
E_{R_Y}(\hat x_{s_n,t}(X)|{\cal F}_r)-\hat \alpha_{r,s_n}(R_Y) \leq \hat x_{r,t}(X).
\label{eqres}
\end{equation}
For a general $r \in [0,T]$, the corresponding equation to (\ref{eqres}) is obtained by right-continuity. 
Then from these equations together with (\ref{eqhatalpha5Yb}), we deduce that 
$\hat x_{r,s}(\hat x_{s,t}(X)) \leq \hat x_{r,t}(X)$. This, together with (\ref{eqtc51}) gives the time-consistency. 

\vspace{1mm}
{\it Step 6: Sandwich and projection property.}

When $r,t \in {\cal T}$, it follows from Theorem  \ref{teo1T}
that $\hat x_{rt}$ extends $x_{rt}$ and satisfies the sandwich condition. These properties extend to all $r \in [0,T]$ making use of the right-continuity and condition mM2 in $x_{r,t}$, $m_{r,t}$ and $M_{r,t}$. They extend then to every $t\in [0,T]$ using  the right-continuity of $m_{r,t}$ and $M_{r,t}$ in $t$ (see mM2), and the fact that  $x_{r,t}$  is the restriction of  $x_{r,T}$ .

For all $r \leq t \leq T$, the projection property for $\hat x_{r,t}$ follows from equation (\ref{eqhatalpha5})
\end{proof}

\begin{corollary}
The  extension $(\hat x_{st})_{s,t \in [0,T]}$ is maximal, in the sense that, for any other such  extension $\big(\bar x_{s,t} \big)_{s,t\in [0,T]}$ we have that: for all $s<t $,
$$
\hat x_{s,t}(X) \geq \bar x_{s,t}(X), \quad X \in L_p(\F_t).
$$
Furthermore  for all $s,t \in [0,T]$, and all $R \in {\cal R}$, $\hat \alpha_{s,t}(R)$ is the minimal penalty associated to $\hat x_{s,t}$, i.e. 
\begin{equation}
\hat \alpha_{s,t}(R)=\esssup_{X \in L_p({\cal F}_t)}[ E_R(X|{\cal F}_s)-\hat x_{s,t}(X)].
\label{eqmp}
\end{equation}
Define for all $R \sim P$ $\hat \alpha_{s,t}(R)$ by the formula (\ref{eqmp}). Then
$(\hat x_{s,t})_{s,t \in [0,T]}$ admits   the following representation where $\hat \alpha_{s,t}$ is the minimal penalty:
\begin{equation}
\hat x_{s,t}(X)=\esssup_{Q \sim P} [E_Q(X|{\cal F}_s) -\hat\alpha_{s,t}(Q)].
\label{eqrs2}
\end{equation}
\end{corollary}

\begin{proof}
The maximality of  $(\hat x_{st})_{s,t \in [0,T]}$ among all the extensions satisfying the required properties follows from Theorem \ref{teo1T}  for $s,t \in {\cal T}$. For $s \in [0,T]$, we apply right-continuity. For $t \in [0,T]$ we apply the fact that $\hat x_{s,t}=\hat x_{s,T}$ and also $\bar x_{s,t}=\bar x_{s,T}$. (see Remark 4.1 and step 1 in the proof of the Theorem).  \\
Note that $\hat x_{s,t}(X+Y)=\hat x_{s,t}(X)+Y$ for all $X \in L_p({\cal F}_t)$ and $Y \in L_p({\cal F}_s)$. This property is known as \emph{$L_p({\cal F}_s)$-translation invariance}. 
 Thus $(\hat x_{st})_{s,t \in [0,T]}$ is up to a minus sign a time-consistent dynamic risk measure. Denote $\beta_{s,t}$  the minimal penalty associated to $\hat x_{s,t}$:
$$
\beta_{s,t}(Q)=\esssup_{X \in L_p({\cal F}_t)}[ E_Q(X|{\cal F}_s)-\hat x_{s,t}(X)].
$$
It follows from Delbaen et al \cite{DRP} appendix, that the minimal penalty $\beta_{s,t}(Q)$ is right-continuous both  in $s$ and $t$, for all $Q \sim P$.
Note that for all $R \in {\cal R}$, $ \hat \alpha_{s,t}(R)$  is also right  continuous both  in $s$ and $t$ (Theorem \ref{teo2T2}). Furthermore we know from Section   \ref{secdiscrete} that $ \hat \alpha_{s,t}(R)$  is the minimal penalty for $\hat x_{s,t}$ for all $s,t \in {\cal T}$. This implies that for all $s,t \in [0,T]$,  and all $R \in {\cal R}$, $ \hat \alpha_{s,t}(R)=  \beta_{s,t}(R)$. 
\end{proof}

\vspace{3mm}
{\bf Acknowledgments.}
The authors thank CMA, University of Oslo, and Ecole Polytechnique, CMAP, Paris for the support in providing occasions of research visits. Also the program SEFE at CAS - Centre of Advanced Study at the Norwegian Academy of Science and Letters is gratefully acknowledged.

\bibliographystyle{plain}

\end{document}